\documentclass[11 pt,fullpage]{article}
\usepackage{amsfonts,amssymb}
\usepackage{graphics,graphicx,pspicture,graphpap}
\usepackage{amsmath,amsthm}
\usepackage{color} 
\usepackage{hyperref}
\usepackage{pstricks,pst-node,pst-text,pst-3d}
\usepackage[right = 2.5cm, left=2.5cm, top = 2.5cm, bottom =2.5cm]{geometry}
\theoremstyle{plain}
\newtheorem{theorem}{Theorem}

\newtheorem{proposition}{Proposition}[section]
\newtheorem{lemma}[proposition]{Lemma}

\newtheorem{definition}{Definition}[section]

\theoremstyle{definition}
\newtheorem{remark}{Remark}

\allowdisplaybreaks

\numberwithin{equation}{section}

\newcommand\R{{\mathbb R}}

\newcommand\Torus{{\mathbb T}}

\def\ff{\widehat{f}}

\def\eps{{\varepsilon}}

\renewcommand\eps{\epsilon }

\newcommand{\Real}{\mathbb R}
\newcommand{\Complex}{\mathbb C}
\newcommand{\Integer}{\mathbb Z}
\newcommand{\norm}[1]{\left\lVert#1\right\rVert}
\newcommand{\abs}[1]{\left\vert#1\right\vert}
\newcommand{\set}[1]{\left\{#1\right\}}

\newcommand{\grad}{\nabla}

\newcommand{\Naturals}{\mathbb N}
\newcommand{\jap}[1]{\langle #1 \rangle}

\newcommand{\dss}{\displaystyle}
\newcommand{\vsp}{\vspace{0.2cm}}

\newcommand\dt{{\frac{\mathrm d}{\mathrm dt}}}

 \newcommand{\dd}{{\, \mathrm d}}

\begin{document}
 
\title{Landau damping in finite regularity for unconfined systems with screened interactions} \author{Jacob
  Bedrossian\footnote{\textit{jacob@cscamm.umd.edu}, University of
    Maryland, College Park. Partially funded by NSF grant DMS-1462029 and an Alfred P. Sloan Research Fellowship}, \, Nader
  Masmoudi\footnote{\textit{masmoudi@cims.nyu.edu}, Courant Institute
    of Mathematical Sciences, New York University. Partially supported
    by NSF grant DMS-1211806} \, and Cl\'ement
  Mouhot\footnote{\textit{c.mouhot@dpmms.cam.ac.uk}, Centre for
    Mathematical Sciences, University of Cambridge. Partially funded
    by ERC grant MATKIT}}

\date{\today}
\maketitle

\begin{abstract} 
We prove Landau damping for the collisionless Vlasov equation with a class of $L^1$ interaction potentials (including the physical case of screened Coulomb interactions) on $\Real^3_x \times \Real^3_v$ for localized disturbances of an infinite, homogeneous background. 
Unlike the confined case $\Torus^3_x \times \Real_v^3$, results are obtained for initial data in Sobolev spaces (as well as Gevrey and analytic classes). 
For spatial frequencies bounded away from zero, the Landau damping of the density is similar to the confined case. 
The finite regularity is possible due to an additional dispersive mechanism available on $\Real_x^3$ which reduces the strength of the plasma echo resonance.
\end{abstract} 

\setcounter{tocdepth}{1}
{\small\tableofcontents}

\section{Introduction} 

\subsection{The model} 
The collisionless Vlasov equation is a fundamental kinetic model for so-called `hot plasmas' and also arises elsewhere in physics, for example, in stellar dynamics \cite{BoydSanderson,LyndenBell67}.  
For single species models, the unknown is the probability density, known as the distribution function, $f(t,x,v)$ of particles in phase space. 
In this work, we consider the phase space $(x,v) \in \Real_x^3 \times \Real_v^3$ and we consider distribution functions of the form $f(t,x,v) = f^0(v) + h(t,x,v)$, where $f^0(v)$ 
is the infinitely extended, homogeneous equilibrium and $h(t,x,v)$ is the mean-zero fluctuation from equilibrium. 
Then, the Vlasov equations for the fluctuation are given by 
\begin{equation}\label{def:VPE}
  \left\{
\begin{array}{l} \dss 
\partial_t h + v\cdot \grad_x h + F(t,x)\cdot \grad_v \left(f^0 + h\right)  = 0, \\[3mm]
F(t,x) := -\grad_x W \ast_{x} \rho(t,x), \\[1mm] 
\dss \rho(t,x) := \int_{\R^d} h(t,x,v) \dd v, \\[3.5mm] 
h(t=0,x,v) = h_{\mbox{{\scriptsize in}}}(x,v).
\end{array}
\right.
\end{equation}
The potential $W(x)$ describes the mean-field interaction between particles.
In this paper we will be considering only $W \in L^1$ which satisfy (denoting $\jap{x} = (1 + \abs{x}^2)^{1/2}$), 
\begin{align}
\abs{\widehat{W}(k)} \lesssim \jap{k}^{-2}. \label{def:W} 
\end{align}
As we will see, one of the reasons for this assumption is that, together with a stability condition involving $f^0$ (see Definition \ref{def:L} below), \eqref{def:W} ensures that the linearized Vlasov equation behaves similarly to the free transport $\partial_th + x \cdot \grad_xh = 0$ (for long times) even at low spatial frequencies. Indeed, the results of \cite{Glassey94,Glassey95} show that this is not true in general if one allows Coulomb interactions $\widehat{W}(k) = \abs{k}^{-2}$. 

Screened Coulomb interactions provide a physically relevant setting which satisfies hypothesis \eqref{def:W} and the stability condition in Definition \eqref{def:L} for a large class of $f^0$ (see Proposition \ref{prop:electrostatic} below). 
This model arises when considering the distribution function for ions in a plasma, after making the approximations of (1) that the electrons can be considered massless and reach thermal equilibrium on a much faster time scale than the ion evolution (2) that the plasma is near equilibrium (3) that an electrostatic approximation is suitable, and (4) that ion collisions can be neglected.  
In this case, the force field $F$ satisfies (some physical parameters have been suppressed for notational convenience)
\begin{align}
F = - \grad \phi, \quad\quad -\Delta \phi + \alpha \phi  = \rho,  \label{def:screenCoulomb}
\end{align} 
where the parameter $\alpha > 0$ accounts for the fact that the electrons equilibrate in a manner to shield the long-range effects of the electric field. 
The quantity $\alpha^{-1/2}$ has units of length and is proportional to the quantity known in plasma physics as the \emph{Debye length}; it is the characteristic length-scale of the mean-field interactions \cite{BoydSanderson}.  
See \cite{HanKwan11,HanKwanIacobelli14,HanKwanRousset15} and the references therein for more details on the model \eqref{def:VPE} with \eqref{def:WscrnC} in the context of ion dynamics in quasi-neutral plasmas.  
In the case of \eqref{def:screenCoulomb}, we have $F = -\grad_x W \ast_x \rho$ with
\begin{align}
\widehat{W}(k) & = \frac{1}{\alpha + \abs{k}^2},  \label{def:WscrnC}
\end{align}
which satisfies \eqref{def:W}. 

\subsection{Landau damping and existing results}

It was discovered by Landau~\cite{Landau46} that the linearized Vlasov equations around homogeneous steady states satisfying certain stability conditions induce time decay on the non-zero modes of the spatial density. 
This decay, which is exponentially fast for analytic data, can be more easily deduced for the free transport evolution $\partial_th + v\cdot \grad_xh = 0$. 
For the free transport evolution, it becomes evident that the decay is due to to \emph{mixing in phase space}, that is, spatial information is transferred to smaller scales in velocity, which are averaged away by the velocity integral for $\rho$ (this appears to be first pointed out in \cite{VKampen55}).
The work of Landau can be summarized as asserting that the dynamics of the linearized Vlasov equations $\partial_th + v\cdot \grad_xh + F\cdot \grad_v f^0 = 0$ are asymptotic to free transport in a suitably strong sense as $t \rightarrow \infty$. 
A number of other works regarding the linearized Vlasov equations followed, providing mathematically rigorous treatments, clarifications, and generalizations \cite{VKampen55,backus,Penrose,Maslov,Degond86,Glassey94,Glassey95}. 
The phenomenon is now known as \emph{Landau damping} and is a cornerstone of plasma physics in approximately collisionless regimes; see e.g. \cite{Ryutov99,BoydSanderson,Stix}.

 The dynamics for each spatial mode decouples in the linearized Vlasov equations and the damping is derived in a relatively straightforward manner via the Laplace transform.  
In the nonlinear equations, there exist steady states and traveling waves with non-trivial densities \cite{BGK57,LZ11b}, however, one can still hope for Landau damping in a perturbative nonlinear regime.   
In the perturbative nonlinear setting, the decoupling of Fourier modes of course ceases to hold and it remained debated for decades
whether or on which timescale the damping would hold (for example, the various discussions in \cite{Penrose,backus,Stix}; see \cite{MouhotVillani11} for more information).
The existence of analytic Landau damping solutions to the nonlinear Vlasov equations in $\Torus_x \times \Real_v$ was first demonstrated in \cite{CagliotiMaffei98,HwangVelazquez09},  
 but only in \cite{MouhotVillani11} was there given a full proof of nonlinear stability
with Landau damping in the nonlinear setting, and again in the confined case $\Torus_x^d \times \Real_v^d$ and for smooth enough Gevrey \cite{Gevrey18} or analytic data. 
The proof was later simplified and the result
improved to the `critical' Gevrey regularity in \cite{BMM13} by combining ideas of \cite{MouhotVillani11} and \cite{BM13}. 

It is desirable for physical relevance to extend the theory to the \emph{unconfined case}, i.e. when the phase space is $\R^d_x \times \Real^d_v$.
There are several issues with this even at the linear level. 
First, at low spatial frequencies, the decay due to mixing for free transport is very slow -- there is an additional dispersive decay but, this is only $t^{-d}$ in $L^\infty$. 
Second, for Vlasov-Poisson, e.g. when the force field is given by $F = -\mu\grad_x \Delta_x^{-1} \rho$ with $\mu \in \Real$, it was shown in \cite{Glassey94,Glassey95} that 
the linearized Vlasov equations cannot be treated as a perturbation of free transport at low spatial frequencies.   
At the linear level, the modes decouple so these issues only occur at low spatial modes; at higher spatial modes, the damping is the same as in the confined case.  
It is then natural to ask whether or not nonlinear stability in \eqref{def:VPE} still holds in a certain sense and that, at least, the decay of the spatial modes away from zero (short waves) remains similar to the confined case.  
In this paper, we positively answer this question in the case that $W$ satisfies \eqref{def:W} (and the linear stability condition in Definition \ref{def:L} below). 
These conditions precisely imply that the linearized Vlasov equation is close enough to free transport at low frequencies.  
Moreover, by taking advantage of a dispersive effect in \emph{frequency} (see \S\ref{sec:Echo} and \S\ref{sec:Proofa}), we are able to get results in finite regularity.

Previous finite regularity Landau damping results have only been obtained for kinetic models in which $\widehat{W}$ has compact support, such as Vlasov-HMF \cite{MR3437866} or the mean-field Kuramoto model \cite{MR3471147,FGVG}. 
These results have been proved in the confined case; see \S\ref{sec:Echo} for more discussion on how finite regularity is obtained. 
A dispersive result in finite-regularity for Vlasov-Poisson in the unconfined case $\Real_x^3 \times \Real_v^3$ without an infinite background density, that is $f^0(v) = 0$, was carried out in \cite{BardosDegond1985}. 
The lack of an infinite background greatly simplifies the setting: the dynamics do not include the linearized Vlasov equations and moreover, it is significantly easier to propagate moments in $x-tv$ on $f(t,x,v)$ -- an important aspect of \cite{BardosDegond1985} (propagating such moments seems very difficult -- possibly impossible -- even for the linearized Vlasov equations with $W \in L^1$ and $f^0$ very small). 
Moreover, the results of \cite{BardosDegond1985} do not directly extend to statements of the form \eqref{ineq:HiFreqDamping} or \eqref{ineq:HiFreqDampingA}, which quantify the fast decay of higher spatial modes (almost equivalently, the techniques seem ill-suited for deducing convergence in such strong norms as \eqref{ineq:glidingconverg} and \eqref{ineq:glidingconvergA}). 

\subsection{The main results}

Our working norm in this paper is the weighted Sobolev norm:  
\begin{align*}
  \norm{h}_{H_M^\sigma} = \sum_{\abs{\alpha} \leq
  M}\norm{\jap{\grad_{x,v}}^\sigma \left(v^\alpha h\right)}_{L^2}. 
\end{align*}

Notice that 
\begin{align*}
  \norm{h}_{H_M^\sigma} \approx_M \left(\sum_{\abs{\alpha} \leq
  M}\norm{\jap{\grad_{x,v}}^\sigma \left(v^\alpha h\right)}_{L^2}^2
  \right)^{1/2} \approx_{M,\sigma} \sum_{\abs{\alpha} \leq M}\norm{v^\alpha \jap{\grad_{x,v}}^\sigma h}_{L^2},  
\end{align*} 
so that one may order the moments and derivatives in whichever order
is most convenient.

The following linear stability condition is essentially an adaptation to finite regularity of the condition given in \cite{MouhotVillani11} (which is essentially the same as the Penrose condition \cite{Penrose}). 

\begin{definition} \label{def:L} 
  Given a homogeneous distribution $f^0(v)$ we say that it satisfies the
  \emph{stability condition \textbf{(L)}} if there exists constants
  $C_0,\kappa,\bar{\sigma} > 0$, with $\bar{\sigma} > 3/2$ and an
  integer $M > 3/2$ such that
\begin{align} 
\norm{f^0}_{H^{\bar{\sigma}}_M} \leq C_0, \label{ineq:f0Loc}
\end{align} 
and
\begin{align} 
  \inf_{\xi \in \Complex: \textbf{Re}\, \xi \leq 0} \inf_{k \in \Real^3}\abs{\mathcal{L}(\xi,k) - 1} \geq \kappa, \label{ineq:kapL}
\end{align} 
where $\mathcal{L}$ is defined by the following ($\bar{\xi}$ denotes the complex conjugate of $\xi$),
\begin{align} 
  \mathcal{L}(\xi,k) = -\int_0^\infty e^{\bar{\xi}t} \widehat {f^0}\left(kt\right) \widehat{W}(k)\abs{k}^2t \dd t. \label{ineq:mL}
\end{align}
\end{definition}

In \S\ref{sec:Linear} below, we discuss in detail how stringent the stability condition \textbf{(L)} is.  
We note here that if one takes power-law interactions, $W(x) = \mu \abs{x}^{-1}$ for any $\mu \in \Real$, then \textbf{(L)} fails for \emph{every} equilibrium $f^0 \in H^{3/2+}_2$, see \cite{Glassey94,Glassey95} (see \S\ref{sec:Note} for the notation $H^{p+}$). 
A smallness condition on $\norm{W}_{L^1} \norm{f^0}_{H^{3/2+}_2}$ is sufficient to
satisfy \eqref{ineq:kapL}, however, it is not necessary.  Indeed, we
show in Proposition \ref{prop:electrostatic} below that \textbf{(L)}
is satisfied for the screened Coulomb law \eqref{def:WscrnC}, the fundamental solution to \eqref{def:screenCoulomb}, for all $\alpha > 0$ and all rapidly decaying, radially symmetric equillibria $f^0$. 
The proof extends to any potential $\widehat{W}$ satisfying 
\begin{align*}
0 \leq \widehat{W}(k) \lesssim \jap{k}^{-2},   
\end{align*}
and hence, a variety of large $W$ and $f^0$ are permitted. 

Our main result is:
\begin{theorem} \label{thm:main} 
There exist universal constants $R_0 > 0$ and $c \in (0,R_0)$ such that if  $\bar{\sigma}-3 > \sigma > R_0$ 
and $f^0$ is given which satisfies stability condition \textbf{(L)} with constants $M$, $C_0$, $\kappa$, and $\bar{\sigma}$ and $h_{\mbox{{\scriptsize {\em in}}}}$ is mean-zero and satisfies
\begin{align*} 
\sum_{\abs{\alpha} \leq 2}\norm{z^\alpha h_{\mbox{{\scriptsize {\em in}}}}}_{H^\sigma_M} \leq \epsilon_0,
\end{align*} 
then there exists a mean-zero $h_\infty \in H^{\sigma-c}_M$ so that the solution $h(t,x,v)$ to
\eqref{def:VPE} satisfies the following for all $t \geq 0$,
\begin{subequations} \label{ineq:damping}
\begin{align} 
  \norm{h(t,x+tv,v) - h_\infty(x,v)}_{H^{\sigma-c}_M} & \lesssim \frac{\epsilon}{\jap{t}^{3/2}}, \label{ineq:glidingconverg} \\
  \abs{\hat{\rho}(t,k)} & \lesssim \epsilon \jap{k,kt}^{-(\sigma-c)},  \label{ineq:HiFreqDamping} \\ 
  \norm{\jap{\grad_x}^{\sigma-c-4} F(t)}_{L^\infty} & \lesssim \frac{\epsilon}{\jap{t}^{4}}. \label{ineq:ErhoLinfty}
\end{align}  
\end{subequations}
\end{theorem} 
\begin{remark} 
The proof shows that we may take $R_0 = 36$ and $c = 5$, although these are unlikely to be sharp. 
\end{remark} 

\begin{remark} 
  Theorem \ref{thm:main} holds in all $d \geq 3$; in this case, $R_0$
  depends in general on dimension.
\end{remark} 

\begin{remark}
  An easier variant of our proof would yield a similar result in the
  case where $f^0=0$ (no homogeneous background). The linear stability
  condition is trivially satisfied then, and our nonlinear estimates
  adapt in a simpler way. 
\end{remark}

\begin{remark} 
That $c$ and $R_0$ are taken independent of all parameters shows that regularity loss remains uniform even as $\sigma \rightarrow \infty$. 
\end{remark}

A natural question is whether or not one still observes exponential
decay of $\hat{\rho}(t,k)$ for $k$ bounded away from zero if the
initial data is analytic.  This is indeed the case, which is proved
via an easy variation of the proof of Theorem \ref{thm:main} using
some basic ideas from \cite{BMM13}.

\begin{theorem} \label{thm:analytic} Let $f^0$ be given which
  satisfies stability condition \textbf{(L)} with constants $M$,
  $C_0$, $\kappa$, and is real analytic with
  \[
  \norm{e^{\bar{\lambda}\jap{\grad}} f^0}_{L^2_M} < \infty \ \mbox{
    for some } \ \bar{\lambda} > 0.
  \]
  Then there exists a $\lambda^\star \in(0,\bar{\lambda}]$ depending
  only on $f^0$ such that for all
  $0 < \lambda^\prime < \lambda < \lambda^\star$, there exists an
  $\epsilon_0$ such that if $h_{\mbox{{\scriptsize {\em in}}}}$ is mean-zero
  and satisfies
  \begin{align*} 
    \sum_{\abs{\alpha} \leq 2}\norm{z^\alpha e^{\lambda \jap{\grad}}
    h_{\mbox{{\scriptsize {\em in}}}}}_{L^2_M} \leq \epsilon_0,
  \end{align*} 
  then there exists a mean-zero, real analytic $h_\infty$ satisfying
  for all $t \geq 0$,
  \begin{subequations} \label{ineq:dampingA}
    \begin{align} 
      \norm{e^{\lambda^\prime \jap{\grad}}\left(h(t,x+tv,v) - h_\infty(x,v)\right)}_{L^2_M} & \lesssim \frac{\epsilon}{\jap{t}^{3/2}}, \label{ineq:glidingconvergA} \\
      \abs{\hat{\rho}(t,k)} & \lesssim \epsilon e^{-\lambda^\prime\jap{k,kt}},  \label{ineq:HiFreqDampingA} \\ 
      \norm{e^{\lambda^\prime\jap{\grad_x,t\grad_x}} F(t)}_{L^\infty} & \lesssim \frac{\epsilon}{\jap{t}^{4}}. \label{ineq:ErhoLinftyA}
    \end{align}  
  \end{subequations}
\end{theorem} 

\begin{remark} 
  An analogue of Theorem \ref{thm:analytic} also holds for Gevrey
  initial data (see \cite{BMM13} for the Vlasov-Poisson systems with
  Coulomb-Newton potentials on $\Torus^d \times \Real^d$ with Gevrey
  data).
\end{remark}

\subsection{Plasma echoes and dispersion in frequency} \label{sec:Echo} 
As discussed in \cite{MouhotVillani11,BMM13}, the fundamental impediment
to nonlinear Landau damping results in finite regularity are resonances known as plasma echoes, first discovered and isolated in the experiments \cite{MalmbergWharton68}. 
During Landau damping, the force field is damped due to the transfer of $O(1)$ spatial information to small scales in the velocity distribution. 
However, mixing is time-reversible, and hence un-mixing creates (transient) growth in the force field.
This effect is essentially the same as the analogous \emph{Orr mechanism} in fluid mechanics, first identified in \cite{Orr07} (see \cite{BM13} for more information).    
A plasma echo occurs when a nonlinear effect transfers information to modes which are \emph{un-mixing}, 
as this leads to a large force field in the future when that information reaches $O(1)$ spatial scales (hence `echo').  
The plasma echo is a kind of nonlinear resonance, although associated with the transient un-mixing in the linear problem rather than a true eigenvalue.  
These echoes can chain into a cascade, as demonstrated experimentally in the Vlasov equations
\cite{MalmbergWharton68} and 2D Euler \cite{YuDriscollONeil,YuDriscoll02}. 

Mathematically, one must confront the echo resonance when attempting to close an estimate such as \eqref{ineq:HiFreqDamping}. 
During the proof of \eqref{ineq:HiFreqDamping}, one needs to get an $L^2_t L^2_k \rightarrow  L_t^2 L_k^2$ estimate on an integral operator that encodes 
the long-time interactions between the force field and the information that has already mixed (see \S\ref{sec:Proofa}). 
The primary new insight in our work is that, unlike in the confined case studied in \cite{CagliotiMaffei98,HwangVelazquez09,MouhotVillani11,BMM13} we can obtain these estimates in \emph{finite regularity}. 
 This is completely due to a dispersive mechanism which is present only in $\Real^d_x$ for $d \geq 2$
(although it is too weak in $d = 2$ for our methods); it has little
relevance to the periodically confined case $\Torus^d_x$ (although one
could imagine attempting to recover it in a large box limit )
and is quite distinct from the finite regularity results of
\cite{MR3437866,FGVG,MR3471147}.

We will capitalize on a dispersive effect in the free transport operator $\partial_t + v\cdot \grad_x$,
which on the \emph{Fourier side} is of the form $\partial_t -k \cdot \grad_\eta$.  
In order to lose a significant amount of regularity, one must chain a large number of echoes over a long period of time (see \cite{MouhotVillani11,BMM13}). Indeed, this is precisely why the finite results of \cite{MR3437866,FGVG,MR3471147} are possible: the models studied therein do not support infinite chains of echoes.  
For any spatial mode $k$, the set of possible ``resonant''frequencies $\ell$, the frequencies which can react strongly via a plasma echo with $k$, turns out to be those which are co-linear with $k$.  
Indeed if the two spatial modes are not co-linear, then the velocity information in the two modes are moving in different directions in frequency (due to the dispersive effect of $\partial_t -k \cdot \grad_\eta$) and is hence well-separated (in frequency) except for a limited amount of time.  
On the torus, the set of such resonant frequencies is of positive density in the lattice $\Integer^d$ (for
example, it suffices to consider modes which depend on only one
coordinate), whereas, in $\Real^3_x$ the set of resonant frequencies is
a one-dimensional line, and is hence a very small set.  
Spatial localization implies that information in the Fourier transform cannot concentrate on small sets, which suggests that the resonance is weaker in $\Real^3$ than in $\Torus^3$. 
This is indeed the case, as we show in \S\ref{sec:Proofa}. 
We remark that there may also be a link with the idea of space-time resonances in dispersive equations \cite{GMS12}.

\subsection{Notation and conventions} \label{sec:Note}
We denote $\Naturals = \set{0,1,2,\dots}$ (including zero) and
$\Integer_\ast = \Integer \setminus \set{0}$.  For $\xi \in \Complex$
we use $\bar{\xi}$ to denote the complex conjugate.  
We denote
\begin{align*} 
\jap{v} = \left( 1 + \abs{v}^2 \right)^{1/2}. 
\end{align*} 
We use the multi-index notation: given
$\alpha = (\alpha_1,\dots,\alpha_d) \in \Naturals^d$ and
$v = (v_1,\dots,v_d) \in \Real^d$ then
\begin{align*} 
  v^\alpha  = v^{\alpha_1}_1v^{\alpha_2}_2 \dots v^{\alpha_d}_d,
  \quad\quad\quad D_\eta^\alpha  = (i\partial_{\eta_1})^{\alpha_1} 
  \dots (i\partial_{\eta_d})^{\alpha_d}. 
\end{align*}  
We denote Lebesgue norms for $p,q \in [1,\infty]$ and $a,b \in \R^3$ as
\begin{align*}  
\norm{f}_{L_a^p L_b^q} = \left(\int_a \left(\int_b \abs{f(a,b)}^q \,
  {\rm d}b \right)^{p/q} \, {\rm d}a\right)^{1/p}
\end{align*}
and Sobolev norms (usually applied to Fourier transforms) as
\begin{align*} 
\norm{\hat{f}}^2_{H^M_\eta} = \sum_{\alpha \in \Naturals^d; \abs{\alpha} \leq M} \norm{D_\eta^\alpha \hat{f}}_{L^2_\eta}^2. 
\end{align*}
We will often use the short-hand $\norm{\cdot}_2$ for
$\norm{\cdot}_{L^2_{z,v}}$ or $\norm{\cdot}_{L^2_v}$ depending on the
context.
Finally we use the notation $f \in H^{s+}$ as short-hand to denote that $f \in H^{s+\delta}$ for all $\delta > 0$. 
Similarly, the quantity $\norm{f}_{H^{s+}}$ is meant to satisfy $\norm{f}_{H^{s+}} \lesssim_\delta \norm{f}_{H^{s+\delta}}$ for all $\delta > 0$ (where the constant in general depends badly on $\delta$ for $\delta \rightarrow 0$).  

For a function $g=g(z,v)$ 
we write its Fourier transform $\hat{g}_k(\eta)$ where
$(k,\eta) \in \R^3 \times \R^3$ with
\begin{align*} 
\hat{g}_k(\eta) := \frac{1}{(2\pi)^{3}}\int_{\R^3 \times \R^3}
  e^{-i z k - iv\eta} g(z,v) \dd z \dd v, \quad 
g(z,v) := \frac{1}{(2\pi)^{3}}\int_{\R^3 \times \R^3}
e^{i z k + iv\eta} \hat{g}_k(\eta) \dd k \dd \eta. 
\end{align*} 
We use an analogous convention for Fourier transforms to functions of
$x$ or $v$ alone. With these conventions we have the following
relations:
\[
\begin{cases}\dss
\int_{\R^3 \times \R^3} g(z,v) \overline{g}(z,v) \dd z \dd v  \dss =
\int_{\R^3 \times \R^3}
\hat{g}(k,\eta) \overline{\hat{g}}(k,\eta) \dd k\dd \eta, \vsp \\[-0.1cm] \dss
\widehat{g^1g^2}  = \dss \frac{1}{(2\pi)^{3}}\widehat{g^1} \ast \widehat{g^2},
\vsp \\[0.1cm] \dss
(\widehat{\grad g})(k,\eta) = (ik,i\eta)\hat{g}(k,\eta),
\vsp  \\[0.2cm] \dss
(\widehat{v^\alpha g})(k,\eta) = D_\eta^\alpha \hat{g}(k,\eta). 
\end{cases}
\]
By convention, we use Greek letters such as $\eta$ and $\xi$ to denote velocity frequencies and lowercase Latin characters such as $k$ and $\ell$ to denote spatial frequencies.

We use the notation $f \lesssim g$ when there exists a constant
$C > 0$ independent of the parameters of interest such that
$f \leq Cg$ (we analogously define $f \gtrsim g$).  Similarly, we use
the notation $f \approx g$ when there exists $C > 0$ such that
$C^{-1}g \leq f \leq Cg$.  We sometimes use the notation
$f \lesssim_{\alpha} g$ if we want to emphasize that the implicit
constant depends on some parameter $\alpha$.

\section{Outline of the proof} 

\subsection{Local-in-time well-posedness}

The following standard lemma provides local existence of a classical
solution which remains classical as long as a suitable Sobolev norm
remains finite.  The propagation of regularity can be proved by a
variant of the arguments in, e.g. \cite{LevermoreOliver97}, along with the
inequality
\begin{align}
  \norm{B(t,\grad_x,\grad_x t)\rho(t)}_2 \lesssim 
  \sum_{\alpha \leq M} \norm{v^\alpha B(t,\grad_x,\grad_v)h(t)}_2,  
\end{align} 
for all Fourier multipliers $B$ and all integers $M > 3/2$. 

\begin{lemma}[Local existence and propagation of
  regularity] \label{lem:loctheory} 
Let $M > 3/2$ be an integer and
  $h_{\mbox{{\scriptsize {\em in}}}} \in H^{\alpha}_M$ for $\alpha > 4$.
  Then there exists some $T_0 > 0$ such that for all $T < T_0$, there
  exists a unique solution $g(t) \in C([0,T];H^{\alpha}_M)$ to
  \eqref{def:VPEgliding} on $[0,T]$.  Moreover, if for some
  $T \leq T_0$ and $\sigma^\prime$ with
  $\sigma \geq \sigma^\prime > 4$, there holds
  $\limsup_{t \nearrow T} \norm{g(t)}_{H^{\sigma^\prime}_{M}} <
  \infty$, then $T < T_0$.
\end{lemma}

\begin{remark} 
Finite energy, strong solutions are well-known to be global in time in $\Torus^3_x \times \Real^3_v$ or on $\Real^3_x \times \Real^3_v$ if there is no homogeneous background \cite{Schaeffer91,Pfaffelmoser92,LP91,Horst93,Benachour89} however, to the authors' knowledge, there is no global existence theory which covers the entire range of Theorem \ref{thm:main}. However, Theorem \ref{thm:main} shows that in the perturbative regime, solutions are global. 
\end{remark} 

\subsection{Coordinate shift} 
As the solution in Theorem \ref{thm:main} is asymptotic to free
transport it makes sense to begin (as in \cite{CagliotiMaffei98,HwangVelazquez09,BMM13}) by modding out
by this evolution:
\begin{subequations} \label{eq:zf}
\begin{align} 
z &:= x-tv \\ 
g(t,z,v) & := h(t,z+tv,v). 
\end{align}
\end{subequations}
From \eqref{eq:zf} and \eqref{def:VPE} we derive the system
\begin{equation}\label{def:VPEgliding}
\left\{
\begin{array}{l}
\partial_t g + F(t,z+vt) \cdot (\grad_v -t\grad_z)g + F(t,z+vt)\cdot \grad_vf^0 = 0, \\[2mm]
g(t = 0,z,v) = h_{\mbox{{\scriptsize in}}}(z,v) \\[2mm] 
\hat{\rho}(t,k) =  \hat g(t,k,kt). 
\end{array}
\right.
\end{equation} 
As in \cite{BMM13}, we derive from \eqref{def:VPEgliding} the following system on the Fourier side, 
\begin{subequations} \label{def:rhof}
\begin{align} 
  \partial_t  \hat g(t,k,\eta) & = 
-\hat{\rho}(t,k) \widehat{W}(k) k \cdot (\eta - tk) \hat{f}^0(\eta - kt) \nonumber \\ & \quad - \int_{\ell} \hat{\rho}(t,\ell)\widehat{W}(\ell)\ell\cdot \left[ \eta  - tk\right] \hat g(t,k-\ell,\eta - t\ell)  \dd\ell \label{eq:feqn} \\
  \hat{\rho}(t,k) &= \hat f_{\mbox{{\scriptsize in}}}(k,kt)   - \int_0^t\hat{\rho}(\tau,k) \widehat{W}(k)k \cdot k(t - \tau) f^0\left(k(t-\tau)\right) \dd\tau \nonumber \\ & \quad
                                                                                                                                                                            - \int_0^t \int_{\ell} \hat{\rho}(\tau,\ell)\widehat{W}(\ell)\ell \cdot k\left(t-\tau\right) \hat g(\tau,k-\ell,kt - \tau \ell)  \dd\ell \dd\tau. \label{eq:rhointegral}
\end{align} 
\end{subequations}

\subsection{Linear Landau damping in $\Real^3_x \times \Real^3_v$} \label{sec:Linear}
The first step in proving Theorem \ref{thm:main} is understanding the linear term in \eqref{eq:rhointegral}.
In particular, we need estimates on the linear Volterra equation 
\begin{align}
  \phi(t,k)= H(t,k) + \int_0^tK^0(t-\tau,k) \phi(\tau,k) \, \dd\tau, \label{eq:volterra}
\end{align}
where $K^0(t,k) := -\hat{f^0}\left(kt\right)\widehat{W}(k)\abs{k}^2t$ and $H(t,k)$ has sufficiently rapid decay. 
Recall that by definition, $\mathcal{L}$ is the Fourier-Laplace transform of the kernel $K^0$:  
\begin{align}
\mathcal{L}(\xi,k) = \int_0^\infty e^{\bar{\xi}t}K^0(t,k) \dd t = -\int_0^\infty e^{ \bar{\xi} t} \abs{k}^2 t \widehat{W}(k) \widehat{f^0}(kt) \dd t. \label{def:L}
\end{align}

We begin by proving that \textbf{(L)} implies Landau damping for
\eqref{eq:volterra}. See Appendix \ref{apx:lin} for the proof, which
is a variation of the arguments in \cite{MouhotVillani11,BMM13}.

\begin{proposition}[Linear $L^2_t$ control] \label{lem:LinearL2Damping} 
  Let $f^0$ satisfy the condition \textbf{(L)} with constants $C_0,\kappa > 0$. Let $\alpha$, be arbitrary and $s \geq 0$ be an arbitrary integer.  
   Let $H(t,k)$ and $T^\star > 0$ be given such that, if we denote $I = [0,T^\star)$ then we take $H(t,k) = 0$ for $t > T_\star$ and
  \begin{align*} 
   \norm{ \abs{k}^{\alpha}\jap{k,kt}^s H(t,k)}_{L_t^2(I)}^2 < \infty.
\end{align*} 
 Then there exists a constant $C_{LD} = C_{LD}(C_0,s,\bar{\sigma},\kappa)$ such
that the solution $\phi(t,k)$ to the system \eqref{eq:volterra} satisfies the pointwise-in-$k$ estimate, 
\begin{align}
\norm{ \abs{k}^{\alpha} \jap{k,kt}^s \phi(t,k)}_{L_t^2(I)} \leq C_{LD} \norm{ \abs{k}^{\alpha} \jap{k,kt}^s H(t,k)}_{L_t^2(I)}^2.  \label{ineq:LinearCtrl}
\end{align} 
\end{proposition}
\begin{remark} 
As long as condition \textbf{(L)} is satisfied, there is no difference between $x \in \Torus^d$ and $x \in \Real^d$ for the purposes of Proposition \ref{lem:LinearL2Damping}. 
In \cite{Glassey94,Glassey95}, the convergence rates are degraded due to the lack of \textbf{(L)}.  
\end{remark} 
\begin{remark} 
In fact, Proposition \ref{lem:LinearL2Damping}  holds for any $s \in \Real_+$, however, the integer case is simpler. Once the integer case is solved, a decomposition argument based on almost-orthogonality is applied to reach fractional $s$; see e.g. \cite{BMM13} for an analogous argument (although the finite regularity setting is easier). 
\end{remark} 

It is important to discuss how restrictive the linear stability condition \textbf{(L)} is. 
The proofs can be found in Appendix \ref{apx:lin}.  
The first observation is that a smallness condition on the interaction is sufficient to imply stability (this follows immediately from Lemma \ref{lem:Lctrl}).  
\begin{proposition} \label{prop:small}
There exists a universal $c > 0$ such that if $\norm{W}_{L^1}\norm{f^0}_{H^{3/2+}_{2}} < c$, then $f^0$, $W$ satisfy the linear stability condition \textbf{(L)} for some $C_0,\kappa$, and $\bar{\sigma}$. 
\end{proposition}

As discussed above, if one takes the interaction potential $W(x) = \mu\abs{x}^{-1}$ for any $\mu \in \Real$, then \textbf{(L)} fails for every equilibrium considered here \cite{Glassey94,Glassey95}. 
However, the screened Coulomb law \eqref{def:WscrnC} does not have this problem. 
Indeed, we have linear stability for all $\alpha > 0$. 
\begin{proposition} \label{prop:electrostatic}
Let $W$ be \eqref{def:WscrnC}, the fundamental solution to \eqref{def:screenCoulomb}. 
Then for any strictly positive, radially symmetric equilibrium $f^0 \in H^{3/2+}_2$ with $\abs{f^0(v)} + \abs{\grad f^0(v)} \lesssim \jap{v}^{-4}$ for $\abs{v}$ large and all $\alpha > 0$, $W$ and $f^0$ satisfy \textbf{(L)} for some constants $C_0,\kappa,$ and $\bar{\sigma}$. In fact, the same applies to any potential $W$ satisfying 
\begin{align*}
0 \leq \widehat{W}(k) \lesssim \jap{k}^{-2}. 
\end{align*}
\end{proposition} 
\begin{remark} 
Of course, the constant $\kappa$ in \textbf{(L)} depends badly on $\alpha$ as $\alpha \rightarrow 0$. 
\end{remark} 

\subsection{Nonlinear energy estimates} 
Next we set up the continuity argument we use to derive a uniform
bound on $g$ via the system \eqref{def:rhof}.  Define the following,
which is convenient when considering the density: for any $s > 0$,
\begin{align*}
A_s(t,k) = \abs{k}^{1/2}\jap{k, tk}^s. 
\end{align*}
We will employ the same notation for the corresponding Fourier multiplier: 
\begin{align*}
A_s(t,\grad_z) = \abs{\grad_z}^{1/2}\jap{\grad_z, t\grad_z}^s; 
\end{align*}
we hope there will be no confusion. 

Fix regularity levels $\bar{\sigma} > \sigma_4 > \sigma_3 > \sigma_2 > \sigma_1 \geq 11$ and constants $K_i \geq 1$ determined by the proof. Let $I = [0,T^\star]$ be the largest connected interval containing zero such that the following bootstrap controls hold: 
\begin{subequations} \label{ctrl:Boot}
\begin{align}
\norm{\jap{t\grad_z,\grad_v} g(t)}^2_{H^{\sigma_4}_M} & \leq 4 K_1\jap{t}^{5}\epsilon^2 \label{ctrl:HiLocalizedBoot} \\
\norm{A_{\sigma_4}\hat{\rho}}_{L^2_t L^2_k}^2 & \leq 4 K_2\epsilon^2. \label{ctrl:MidBoot}\\
\norm{\abs{\grad_z}^\delta g(t)}^2_{H^{\sigma_3}_M} & \leq 4 K_3\epsilon^2 \label{ctrl:LoLocalizedBoot} \\
\norm{A_{\sigma_2} \hat{\rho}}_{L^\infty_k L^2_t} & \leq 4 K_4\epsilon^2 \label{ctrl:MidLinftyrho} \\
\norm{\widehat{\jap{\grad}^{\sigma_1}g}}_{L^\infty_{k,\eta}} & \leq 4 K_5\epsilon^2. \label{ctrl:LowLinftyf}
\end{align}
\end{subequations}
\begin{remark} 
A close reading of the proof suggests that one can take $\sigma_i - \sigma_{i-1} = 6$ and $\bar{\sigma} - \sigma_4 = 6$, although this seems far from optimal. This technically brings the regularity requirement given by the proof to 35, however, we did not attempt to optimize this number.  
\end{remark} 

\begin{remark} \label{rmk:Ki}
The constants $K_2$ and $K_4$ are determined only by the properties of the linearized Vlasov equations (hence they depend only on $f^0$ and $W$), and the constants $K_1,K_3,K_5$ are fixed independently, depending only on $K_2$, $K_4$, and universal constants.   
\end{remark} 

\begin{remark} 
Notice the order $L^\infty_k L^2_t$ in the estimate \eqref{ctrl:MidLinftyrho}. This  norm is reminiscent of the norms used by Chemin and Lerner in  \cite{CheminLerner95}. 
\end{remark} 

\begin{proposition}[Bootstrap] \label{prop:boot}
Let \eqref{ctrl:Boot} be satisfied for all $t \in [0,T^\star]$ with $T^\star < T^0$ ($T^0$ defined in \eqref{lem:loctheory}). Then for $\epsilon$ chosen sufficiently small, the estimates \eqref{ctrl:Boot} all hold with $4$ replaced with $2$.  
\end{proposition} 

Proposition \ref{prop:boot} comprises the main step of the proof of Theorem \ref{thm:main} (see Proposition \ref{prop:final} below). 

\subsection{Useful toolbox} 
First, we observe the following, which at least shows that the norms employed to measure $\rho$ in \eqref{ctrl:Boot} are natural. 
\begin{lemma} \label{lem:IDconst}
Define 
\begin{align*}
  \rho_0(t,k) = \widehat{h_{\mbox{{\scriptsize {\em in}}}}}(k,kt). 
\end{align*}
For all $s > 4$, there holds (recall the notation $H^{s+}$ from \S\ref{sec:Note}),  
\begin{align*}
\norm{A_s \rho_0}_{L^2_t L^2_k} & \lesssim \sum_{\abs{\alpha} \leq
                                  2}\norm{z^\alpha
                                  h_{\mbox{{\scriptsize {\em in}}}}}_{H^{s+2+}_M} \\ 
\norm{A_s \rho_0}_{L^\infty_k L^2_t} & \lesssim \sum_{\abs{\alpha}
                                       \leq 2}\norm{z^\alpha
                                       h_{\mbox{{\scriptsize {\em in}}}}}_{H^{s+1+}_M}. 
\end{align*}
\end{lemma} 
\begin{proof} 
  The proof of the first estimate is straightforward by the
  $H^{3/2+}(\Real^3) \hookrightarrow C^0$ Sobolev embedding applied on the
  Fourier side and $M \geq 2$,
\begin{align*}
\int_{0}^\infty \int_{k} \abs{k} \jap{k,kt}^{2s}
  \abs{\widehat{h_{\mbox{{\scriptsize in}}}}(k,kt)}^2 \dd k \dd t &
                                                                    \lesssim
                                                                    \left(
                                                                    \int_{0}^\infty
                                                                    \int_{k}
                                                                    \frac{\abs{k}}{\jap{k,kt}^{4+}}
                                                                    \dd
                                                                    k
                                                                    \dd
                                                                    t
                                                                    \right)
                                                                    \left(\sum_{\abs{\alpha} \leq 2}\norm{z^\alpha h_{\mbox{{\scriptsize in}}}}_{H^{s+2+}_M}^2\right) \\ 
& \lesssim \left(1 + \int_{1}^\infty \frac{1}{t^{4}} \int_{\zeta} \frac{\abs{\zeta}}{\jap{\zeta}^{4+}} \dd \zeta \dd t\right) \left(\sum_{\abs{\alpha} \leq 2}\norm{z^\alpha h_{\mbox{{\scriptsize in}}}}_{H^{s+2+}_M}^2\right) \\ 
 & \lesssim \sum_{\abs{\alpha} \leq 2}\norm{z^\alpha h_{\mbox{{\scriptsize in}}}}_{H^{\sigma+2+}_M}^2. 
\end{align*} 
The second estimate follows slightly differently. For all $k \in \Real^3$ we have, 
\begin{align*}
\int_0^\infty \abs{k}
  \jap{k,kt}^{s}\abs{\widehat{h_{\mbox{{\scriptsize in}}}}(k,kt)}^2
  \dd t & \lesssim \left( \int_0^\infty
          \frac{\abs{k}}{\jap{k,kt}^{1+}} \dd t \right) \left(\sum_{\abs{\alpha} \leq 2}\norm{z^\alpha h_{\mbox{{\scriptsize in}}}}_{H^{s+1+}_M}^2\right) \\
& \lesssim \left(\sum_{\abs{\alpha} \leq 2}\norm{z^\alpha h_{\mbox{{\scriptsize in}}}}_{H^{s+1+}_M}^2\right).
\end{align*}
\end{proof} 

Next, let us point out a consequence of the estimate \eqref{ctrl:LowLinftyf}, which provides the dispersive decay of the density and force field.  
\begin{lemma}
Under the bootstrap hypotheses, for all $0 \leq \alpha < \sigma_1 - \gamma - 3$, there holds 
\begin{align}
\norm{\abs{\partial_z}^{\alpha} \jap{\partial_z,t\partial_z}^\gamma \rho}_{L^\infty} \lesssim \int_k  \abs{k}^{\alpha} \jap{k,kt}^\gamma \abs{\widehat{\rho}(t,k)} \dd k \lesssim K_5\epsilon\jap{t}^{-3-\alpha}. \label{ineq:dissipReq} 
\end{align}
\end{lemma} 
\begin{proof} 
This follows immediately from \eqref{ctrl:LowLinftyf} (and recalling that $\hat{g}(t,k,kt) = \hat{\rho}(t,k)$ from \eqref{def:VPEgliding}), 
\begin{align*}
\int_k  \abs{k}^{\alpha} \jap{k,kt}^\gamma \abs{\widehat{\rho}(t,k)} \dd k \lesssim \epsilon \int_k  \abs{k}^{\alpha} \jap{k,kt}^{\gamma-\sigma_1} \dd k \lesssim K_5 \jap{t}^{-\alpha}\epsilon \int_k \jap{k,kt}^{\gamma-\sigma_1 + \alpha} \dd k \lesssim K_5\epsilon\jap{t}^{-3-\alpha}.
\end{align*}
\end{proof} 

Let us also record a few simple inequalities that will be used a few times in the sequel (on the Fourier-side). 
\begin{lemma}[$L^2$ Trace] \label{lem:SobTrace}
Let $g \in H^s(\Real^d)$ with $s > (d-1)/2$ and $C \subset \Real^d$ be an arbitrary straight line. Then there holds, 
\begin{align*} 
\norm{g}_{L^2(C)} \lesssim_s \norm{g}_{H^s}. 
\end{align*} 
\end{lemma} 

\begin{lemma} \label{lem:Youngs} 
\begin{itemize} 
\item[(a)] Let
  $g^1,g^2 \in L^2(\Real^d_k \times \Real^d_\eta)$ and
  $r \in L^1(\Real^d_\eta)$. Then
\begin{align} 
\abs{\int  g^1(k,\eta) r(\ell)
  g^2(k-\ell,\eta-t\ell) \dd\ell \dd k \dd\eta
 } \lesssim \norm{g^1}_{L^2_{k,\eta}}
  \norm{g^2}_{L^2_{k,\eta}} \| r\|_{L^1_\eta}. \label{ineq:L2L2L1}      
\end{align} 
\item[(b)] Let $g^1 \in L^2(\Real^d_k \times \Real^d_\eta)$ and
  $g^2 \in L^1(\Real^d_k; L^2(\Real^d_\eta))$ and
  $r \in L^2(\Real^d)$. Then
\begin{align} 
\abs{\int g^1(k,\eta) r(\ell)
  g^2(k-\ell,\eta-t\ell) \dd\ell \dd k \dd\eta} \lesssim \norm{g^1}_{L^2_{k,\eta}}
  \norm{g^2}_{L^1(\Real^d_k; L^2(\Real^d_\eta))} 
\norm{r}_{L^2_\eta}.    \label{ineq:L2L1L2}      
\end{align} 
\end{itemize}
As a result, if $s > d/2$, there also holds 
\begin{align}
  \abs{\int  g^1(k,\eta) r(\ell) g^2(k-\ell,\eta-t\ell) \dd\ell \dd k \dd\eta} & \lesssim_{d,s} \norm{g^2}_{L^2_{k,\eta}} \norm{g^2}_{L^2_{k,\eta}} \norm{\jap{\cdot}^s r(\cdot)}_{L^2_\eta}    \label{ineq:L2L2L1_L2} \\       
  \abs{\int  g^1(k,\eta) r(\ell) g^2(k-\ell,\eta-t\ell) \dd \ell \dd k \dd \eta} & \lesssim_{d,s} \norm{g^1}_{L^2_{k,\eta}} \norm{\jap{k}^s g^2}_{L^2_{k,\eta}} \norm{r}_{L^2_\eta}.    \label{ineq:L2L1L2_L2}      
\end{align}

\end{lemma} 
 
As a straightforward application of the above lemmas, we show that Proposition \ref{prop:boot} implies Theorem \ref{thm:main}. 

\begin{proposition} \label{prop:final} 
Proposition \ref{prop:boot} implies Theorem \ref{thm:main}.  
\end{proposition} 
\begin{proof} 
Estimate \eqref{ctrl:LowLinftyf} directly implies \eqref{ineq:HiFreqDamping} by $\hat{\rho}(t,k) = \hat{g}(t,k,kt)$ whereas \eqref{ineq:ErhoLinfty} follows by \eqref{ineq:dissipReq} above (also a direct consequence of \eqref{ctrl:LowLinftyf}).  

To deduce \eqref{ineq:glidingconverg}, begin by applying $\jap{k,\eta}^{\sigma_0} D_\eta^\alpha$ for a multi-index $\abs{\alpha} \leq M$ and $\sigma_0 < \sigma_1-5/2$ and integrating \eqref{eq:feqn}: 
\begin{align}
\jap{k,\eta}^{\sigma_0} D_\eta^\alpha \hat{g}(t,k,\eta) & = \jap{k,\eta}^{\sigma_0}D_\eta^\alpha \hat{g}_{\mbox{{\scriptsize in}}}(k,\eta) - \int_0^t\jap{k,\eta}^{\sigma_0} D_\eta^\alpha \left( \hat{\rho}(\tau,k) \widehat{W}(k) k \cdot (\eta - \tau k) \hat{f}^0(\eta - k\tau) \right) \dd\tau \nonumber \\ & \quad - \int_0^t\int_{\ell} \jap{k,\eta}^{\sigma_0} D_\eta^\alpha \left( \hat{\rho}(\tau,\ell)\widehat{W}(\ell)\ell\cdot \left[ \eta  - \tau k\right] \hat g(\tau,k-\ell,\eta - t\ell) \right)  \dd\ell \dd\tau \nonumber \\ 
& = \jap{k,\eta}^{\sigma_0} D_\eta^\alpha \hat{g}_{\mbox{{\scriptsize in}}} - \int_0^t L \dd\tau - \int_0^t NL \dd\tau. \label{eq:final}
\end{align}
By Proposition \ref{prop:boot} there holds (using that $\bar{\sigma}$ is sufficiently large), 
\begin{align*}
\norm{L}_{L^2_{k,\eta}}^2 & \lesssim \int_{k,\eta} \abs{k}^2\jap{k,kt}^{2\sigma_0}\abs{\rho(t,k)}^2 \jap{\eta-kt}^{2\sigma_0} \abs{D_\eta^\alpha \left((\eta-kt)\widehat{f^0}(\eta-kt)\right)}^2 \dd k d\eta \\ 
& \lesssim \int_{k} \abs{k}^2\jap{k,kt}^{2\sigma_0}\abs{\rho(t,k)}^2 \dd k \\ 
& \lesssim \epsilon^2 \int_k \abs{k}^2\jap{k,kt}^{2\sigma_0-2\sigma_1} \dd k \\
& \lesssim \epsilon^2 \jap{t}^{-5}. 
\end{align*}
Similarly, 
\begin{align*}
\norm{NL}_{L^2_{k,\eta}}^2 & \lesssim \jap{t}^2\norm{\abs{\partial_z}^{\delta} f}_{L^{\sigma_0}_M}^2 \left(\int_{\ell} \frac{\abs{\ell}}{\abs{k-\ell}^{\delta}} \jap{\ell,\ell t}^{\sigma_0} \abs{\hat{\rho}(t,\ell)}  \dd\ell \right)^2 \\ 
 & \lesssim \epsilon^4 \jap{t}^{2\delta-6}. 
\end{align*}
Therefore, both of the time-integrals in \eqref{eq:final} are absolutely convergent in $L^2_{k,\eta}$ (recall $0 < \delta \ll 1/2$). 
Hence, define
\begin{align*}
\widehat{h_\infty}(k,\eta) & := \hat{g}_{\mbox{{\scriptsize in}}}(k,\eta) - \int_0^\infty \hat{\rho}(\tau,k) \widehat{W}(k) k \cdot (\eta - \tau k) \hat{f}^0(\eta - k\tau) \dd\tau \\ & \quad - \int_0^\infty \int_{\ell} \hat{\rho}(\tau,\ell)\widehat{W}(\ell)\ell\cdot \left[ \eta  - \tau k\right] \hat g(\tau,k-\ell,\eta - t\ell)  \dd\ell \dd\tau. 
\end{align*}
Inequality \eqref{ineq:glidingconverg} then follows from the decay estimates on the integrands and the definition of $g$. 
\end{proof}

\section{Plasma echoes in finite regularity} \label{sec:Proofa} 
As discussed in \S\ref{sec:Echo}, the plasma echo effect is the main difficulty in deducing Landau damping. 
When attempting the estimate \eqref{ctrl:MidBoot}, one must get an
$L^2_t L^2_k \rightarrow L_t^2 L_k^2$ estimate on the integral operator:
\begin{align}
\phi(t,k) \mapsto \int_0^t \int_\ell \phi(\tau,\ell) \bar{K}(t,\tau,k,\ell) \dd \ell \dd\tau, \label{def:nonint}
\end{align}
where the so-called \emph{time-response kernel} is given by  
\begin{align} 
\bar{K}(t,\tau,k,\ell) & = \frac{\abs{k}^{1/2} \abs{\ell}^{1/2}
                         \abs{k(t-\tau)}}
{\jap{\ell}^2}\abs{\hat{g}(\tau,k-\ell,kt-\ell \tau)}, \label{def:trK}
\end{align}
as will be derived in \S\ref{sec:Ereac} below. This kernel measures
the maximal strength at which the $\ell$-th mode of the density at
time $\tau$ can force the $k$-th mode of the density at time $t$
through the nonlinear interaction with $g$ at mode
$k-\ell,kt-\ell \tau$ at time $\tau$.  By \eqref{ctrl:LowLinftyf}, we estimate
\begin{align}
  \bar{K}(t,\tau,k,\ell) & \lesssim
                           \sqrt{K_2}\epsilon\frac{\abs{k}^{1/2}
                           \abs{\ell}^{1/2} 
                           \jap{\tau}}{\jap{\ell}^2\jap{k-\ell,kt-\ell\tau}^{\sigma_1 - 1}};  \label{def:barK} 
\end{align}
for notational convenience we define $\beta:= \sigma_1 - 1$. 
By Schur's test, it suffices to bound the
supremum of the row-sums and the supremum of the column-sums of
\eqref{def:barK} in order to show that the integral operator
\eqref{def:nonint} is bounded.  This is the content of this
section, proved below in Lemmas \ref{lem:Moment} and \ref{lem:dual}.
Similar time-response kernels arose in \cite{BMM13} and
\cite{MouhotVillani11} -- the primary new insight here is the fact
that we can prove Lemmas \ref{lem:Moment} and \ref{lem:dual} in finite regularity.

It is clear that the row and column sums of \eqref{def:barK} are dominated by contributions from large $\tau$ and where $kt - \ell \tau$ is small, only possible when $k$ and $\ell$ are nearly co-linear.  
On $\Torus^d_x$, the one dimensional reductions used in the proofs of the analogous lemmas in \cite{MouhotVillani11,BMM13} are essentially reductions to co-linear resonant frequencies.  
In the proofs of Lemmas \ref{lem:Moment} and \ref{lem:dual}, we will separate the approximately co-linear `resonant' frequencies from the `non-resonant' frequencies with a time-varying cut-off. 
The fact that we can take the cut-off shrinking in time is due to the dispersion encoded in the free transport on the frequency side, $\partial_t + k\cdot \grad_\eta$.    
We will then use that the `resonant' frequencies comprise a small set which shrinks in time whereas on the `non-resonant' frequencies, $\bar{K}$ has much better estimates. 
The cut-off is then chosen to balance both requirements; it is in this balance where $d \geq 3$ is used.

\begin{lemma}[Time response estimate I] \label{lem:Moment}
Under the bootstrap hypotheses \eqref{ctrl:Boot}, there holds
\begin{align*} 
\sup_{t \in [0,T^\star]} \sup_{k \in \Real^3}\int_0^t \int_{\ell} \bar{K}(t,\tau,k,\ell) \dd\tau  \dd\ell \lesssim \sqrt{K_2} \epsilon.  
\end{align*} 
\end{lemma} 
\begin{proof}
First, we eliminate irrelevant early times: for $\beta > 4$ we have, 
\begin{align*}   
\int_0^{\min(1,t)} \int_{\ell} \frac{\jap{\tau}\abs{\ell}^{1/2} \abs{k}^{1/2}}{\jap{\ell}^2\jap{k-\ell,kt-\ell \tau}^{\beta}}  \dd\ell \dd\tau & \lesssim 1.     
\end{align*} 
For a fixed $k\in \Real^3$ and all $\ell \in \Real^3$ define
\begin{align*} 
\ell_{||} = \frac{k \cdot \ell}{\abs{k}^2} k \\ 
\ell_{\perp} = \ell - \ell_{||},   
\end{align*} 
the co-linear and perpendicular components. 
Define the following parameters
\begin{subequations} \label{def:zetab} 
\begin{align} 
\zeta &\in \left(\frac{4}{5},1\right) \\ 
b & = \beta^{-1},  
\end{align}
\end{subequations}  
where we will choose $\beta$ such that at least $b < 1/6$. 
Define the two subregions of resonant $\ell$ and non-resonant $\ell$: 
\begin{align*} 
I_{R} & = I_{R}(\tau,k)  = \set{\ell: \abs{\ell_{\perp}} < (1+ \tau)^{-\zeta}\abs{k}^b } \\ 
I_{NR} & = I_{NR}(\tau,k) = \set{\ell: \abs{\ell_{\perp}} \geq (1+ \tau)^{-\zeta}\abs{k}^b }.
\end{align*}
The set $I_{R}$ denotes the frequencies that can resonate strongly with frequency $k$ and is a cylinder around the line containing $k$ which is shrinking in time. 
Physically, $I_{R}$ is restricting to the frequencies which spend a long time  sufficiently aligned with $k$.
The dispersive effect is highlighted due to the fact that we can shrink the cross-sectional area of the cylinder in time. 
In $I_{R}$, for each $\ell$ with $\ell= \ell_{||}$ we can associate a disk of radius $(1+ \tau)^{-\zeta}\abs{k}^b$ which lies in the resonant region. 
We first integrate over this 2D disk; this is where we are going to exploit that $\ell \in \Real^d$ with $d = 3$ (note also that we have used the inequality $\abs{x+y}^{1/2} \leq \abs{x}^{1/2} + \abs{y}^{1/2}$ for $x,y > 0$):  
\begin{align} 
\int_{\min(1,t)}^t \int_{I_{R}}  \frac{\jap{\tau}\abs{\ell}^{1/2}\abs{k}^{1/2}}{\jap{\ell}^2 \jap{k-\ell,kt-\ell \tau}^{\beta}}  \dd\ell \dd\tau & \lesssim \int_{\min(1,t)}^t \int_{I_{R}} \frac{\jap{\tau} \abs{k}^{1/2} \left(\abs{\ell_{||}}^{1/2} + \abs{\ell_{\perp}}^{1/2}\right)}{\jap{\ell_{||}}^2 \jap{k-\ell,kt-\ell\tau}^{\beta}}  \dd\ell \dd\tau \nonumber  \\ 
& \hspace{-5cm} \lesssim \int_{\min(1,t)}^t \int_{\Real} \frac{\jap{\tau} \abs{k}^{1/2}}{\jap{\ell_{||}}^2 \jap{k-\ell_{||},k t-\ell_{||}\tau}^{\beta}} \frac{\abs{k}^{2b}}{(1 + \tau)^{2\zeta}}\left(\abs{\ell_{||}}^{1/2} + \frac{\abs{k}^{b/2}}{(1 + \tau)^{\zeta/2}}\right)  \dd\ell_{||} \dd\tau \nonumber \\ 
& \hspace{-5cm} = \mathcal{I}_1 + \mathcal{I}_2. \label{ineq:RR} 
\end{align} 
In particular, $\abs{k}^{2b}(1+\tau)^{-2\zeta} = \abs{k}^{(d-1)b}(1+\tau)^{-(d-1)\zeta}$, and hence the argument extends to all $d \geq 3$.  
Then, requiring that $\zeta > 1/2$, $2b + 1/2 < 2$, and $\beta > 4$ we get 
\begin{align*} 
\mathcal{I}_1 & \lesssim \int_{\min(1,t)}^t \int_{\Real}  \frac{\abs{\ell_{||}}^{1/2}}{\jap{k-\ell_{||},kt-\ell_{||}\tau}^{\beta- 1/2 - 2b}}  \dd\ell_{||} \dd\tau \\ 
&  \lesssim_\beta \int_{\Real}  \frac{1}{\abs{\ell_{||}}^{1/2}\jap{k-\ell_{||}}^{\beta-3}}  \dd\ell_{||} \\ 
& \lesssim_{\beta} 1,   
\end{align*}
which completes the first term in \eqref{ineq:RR}.  
For the second term in \eqref{ineq:RR} we require $\zeta > 4/5$ and $1 + 5b < 4$:    
\begin{align*} 
\mathcal{I}_2 &\lesssim \int_{\min(1,t)}^t \int_{\Real} \frac{\abs{k}^{1/2 + 5b/2}}{\jap{\ell_{||}}^2 \jap{k-\ell_{||},kt-\ell_{||}\tau}^{\beta}\jap{\tau}^{5\zeta/2-1}}  \dd\ell_{||} \dd\tau  \\
 & \lesssim_\beta \int_{\Real} \frac{1}{\jap{k-\ell_{||}}^{\beta-1/2-3b}}   \dd\ell_{||} \\
& \lesssim_\beta 1.  
\end{align*} 
This completes the treatment of the resonant region in \eqref{ineq:RR}. 

Turn next to the $I_{NR}$ region. 
In this region, 
\begin{align*} 
\abs{kt - \ell \tau} \gtrsim \tau \abs{\ell_{\perp}} \geq \frac{\tau}{(1 + \tau)^{\zeta}} \abs{k}^b. 
\end{align*} 
Therefore, using that $b = \beta^{-1}$
\begin{align*} 
\int_{\min(1,t)}^t \int_{I_{NR}} \frac{\abs{k}^{1/2}\abs{\ell}^{1/2}\jap{\tau}}{\jap{\ell}^2 \jap{k-\ell,kt-\ell\tau}^{\beta}}  \dd\ell \dd\tau  &\lesssim \int_{\min(1,t)}^t \int_{I_{NR}} \frac{\jap{\tau}\abs{k}^{1/2}\abs{\ell}^{1/2}}{\jap{\ell}^2 \jap{k-\ell,kt - \ell \tau}^{\frac{\beta}{2}} \abs{k}^{\frac{\beta b}{2}} \jap{\tau}^{\frac{\beta}{2}(1-\zeta)}}  \dd\ell \dd\tau \\ 
& \lesssim \int_{\min(1,t)}^t \int_{I_{NR}}\frac{\abs{\ell}^{1/2}}{\jap{\ell}^2 \jap{k-\ell,kt - \ell \tau}^{\frac{\beta}{2}} \jap{\tau}^{\frac{\beta}{2}(1-\zeta)-1}}  \dd\ell \dd\tau \\ 
& \lesssim \int_{\min(1,t)}^t \int_{\Real^3} \frac{\abs{\ell}^{1/2}}{\jap{k-\ell,kt - \ell \tau}^{\frac{\beta}{2}-1} \jap{\tau}^{\frac{\beta}{2}(1-\zeta)-1}}  \dd\ell \dd\tau.   
\end{align*}
This integral is uniformly bounded provided that (using that the dimension is $3$),  
\begin{align*}
\frac{\beta}{2}-1 & > 4 \\ 
\frac{\beta}{2}(1-\zeta) & > 1, 
\end{align*}
which using that $1-\zeta < 1/5$ (and otherwise $\zeta$ is arbitrary), gives, the regularity requirement, 
\begin{align*}
\beta > 10,  
\end{align*} 
which is also sufficiently large to satisfy all of the other large-ness conditions above as well.  
\end{proof}

The next estimate is in some sense the `dual' of Lemma \ref{lem:Moment} and the proof is analogous.  

\begin{lemma}[Time response estimate II] \label{lem:dual}
Under the bootstrap hypotheses \eqref{ctrl:Boot} there holds
\begin{align*}
\sup_{\tau \in [0,T^\star]} \sup_{\ell \in \Real^d} \int_{k \in \Real^d} \int_{\tau}^{T^\star}\bar{K}_{k,\ell}(t,\tau) \dd t \dd k \lesssim \sqrt{K_2} \epsilon. 
\end{align*}
\end{lemma} 
\begin{proof}
As above, we eliminate irrelevant early times: for $\beta > 4$ we have, 
\begin{align*}
\int_{k \in \Real^d} \int_{\tau}^{\min(1,T^\star)} \frac{\jap{\tau}\abs{k}^{1/2}\abs{\ell}^{1/2}}{\jap{\ell}^2\jap{k-\ell,kt-\ell\tau}^{\beta}} \dd k \dd t & \lesssim 1.  
\end{align*}
For a fixed $\ell \in \Real^3$ and all $k \in \Real^3$ define
\begin{align*} 
k_{||} = \frac{k \cdot \ell}{\abs{\ell}^2} \ell \\ 
k_{\perp} = k - k_{||},   
\end{align*} 
the co-linear and perpendicular components. 
Fix the following parameters as in the proof of Lemma \ref{lem:Moment}. 
\begin{subequations} \label{def:zetab} 
\begin{align} 
\zeta &\in \left(\frac{4}{5},1\right) \\ 
b & = \beta^{-1}. 
\end{align}
\end{subequations}  
Define the two subregions: 
\begin{align*} 
I_{R} & = I_R(t,\ell) = \set{k: \abs{k_{\perp}} < (1+ t)^{-\zeta}\abs{\ell}^b } \\ 
I_{NR} & = I_{NR}(t,\ell) = \set{k: \abs{k_{\perp}} \geq (1+ t)^{-\zeta}\abs{\ell}^b }. 
\end{align*} 
As above, $I_{R}$ is cutting out a shrinking cylinder around the line containing $\ell$ and restricting to the set of frequencies which can create strong echo cascades over the time interval of interest. 
Integrating over the 2D disk as above, 
\begin{align} 
\int_{\min(1,T^\star)}^{T^\star} \int_{I_{R}} \frac{\jap{t}\abs{\ell}^{1/2}\abs{k}^{1/2}}{\jap{\ell}^2 \jap{k-l,kt-l\tau}^{\beta}} \dd k  \dd t & \lesssim \int_{\min(1,T^\star)}^{T^\star} \int_{I_{R}} \frac{\jap{t}\abs{\ell}^{1/2}\abs{k}^{1/2}}{ \jap{\ell}^2 \jap{k_{||}-\ell,k_{||}t-\ell \tau}^{\beta}} \dd k \dd t \nonumber  \\
& \hspace{-5cm} \lesssim \int_{\min(1,T^\star)}^{T^\star} \int_{\Real} \frac{\jap{t} \abs{\ell}^{1/2} }{ \jap{\ell}^2 \jap{k_{||}-\ell,k_{||}t - \ell \tau}^{\beta}} \frac{\abs{\ell}^{2b}}{(1 + t)^{2\zeta}}\left(\abs{k_{||}}^{1/2} + \frac{\abs{\ell}^{b/2}}{(1 + t)^{\zeta/2}}\right) d k_{||} \dd t \nonumber \\ 
& = \mathcal{I}_1 + \mathcal{I}_2.  \label{ineq:RR2} 
\end{align} 
Then using that $\zeta > 1/2$ and $\beta$ sufficiently large (equivalently, $b$ sufficiently small), 
\begin{align*} 
\mathcal{I}_1 & \lesssim \int_{\min(1,T^\star)}^{T^\star} \int_{\Real}  \frac{\abs{k_{||}}^{1/2}}{ \jap{k_{||}-\ell,k_{||}t - \ell \tau}^{\beta-1}} \dd k_{||} \dd t \\ 
&  \lesssim_b \int_{\Real}  \frac{1}{\abs{k_{||}}^{1/2} \jap{k_{||} - \ell}^{\beta-3}} \dd k_{||} \\ 
&  \lesssim_{\beta} 1, 
\end{align*} 
For the other contribution in \eqref{ineq:RR2} we use $\zeta > 4/5$ and $\beta$ sufficiently large (equivalently, $b$ sufficiently small), 
\begin{align*} 
\mathcal{I}_2 & \lesssim \int_{\min(1,T^\star)}^{T^\star} \int_{\Real} \frac{1}{ \jap{k_{||} - \ell,k_{||}t - \ell \tau}^{\beta-1-3b} (1 + t)^{5\zeta/2-1}} \dd k_{||} \dd t  \\
 & \lesssim \int_{\Real} \frac{1}{ \jap{k_{||} - \ell}^{\beta-1-3b}}  \dd k_{||} \\
& \lesssim 1.  
\end{align*} 
This completes the treatment of the resonant region in \eqref{ineq:RR2}. 

Turn now to the non-resonant $I_{NR}$ region. 
On the support of the integrand, notice that
\begin{align*} 
\abs{kt - \ell \tau} \gtrsim t \abs{k_{\perp}} \geq \frac{t}{(1 + t)^{\zeta}} \abs{\ell}^b. 
\end{align*} 
Recalling the choice $b = \beta^{-1}$ we get, 
\begin{align*} 
\int_{\min(1,T^\star)}^{T^\star} \int_{I_{NR}} \frac{\abs{k}^{1/2}\abs{\ell}^{1/2}\jap{t}}{\jap{\ell}^2 \jap{k-\ell,kt-\ell\tau}^{\beta}} \dd k \dd t  &\lesssim \int_{\min(1,T^\star)}^{T^\star} \int_{I_{NR}} \frac{\jap{t}\abs{k}^{1/2}\abs{\ell}^{1/2}}{\jap{\ell}^2\jap{k-\ell,kt - \ell \tau}^{\frac{\beta}{2}} \abs{\ell}^{1/2} \jap{t}^{\frac{\beta}{2}(1-\zeta)}} \dd k \dd t \\ 
& \lesssim \int_{\min(1,T^\star)}^{T^\star} \int_{I_{NR}} \frac{\abs{k}^{1/2}}{\jap{\ell}^{3/2}\jap{k-\ell,kt - \ell \tau}^{\frac{\beta-1}{2}} \jap{t}^{\frac{\beta}{2}(1-\zeta)-1}} \dd k \dd t. 
\end{align*} 
This integral is uniformly bounded in $\ell$ and $\tau$ if (using that the dimension is $3$),  
\begin{align*}
\frac{\beta}{2} - 1 & > 4 \\ 
\frac{\beta}{2}(1-\zeta) & > 1,
\end{align*}
as in Lemma \ref{lem:Moment} above. 
\end{proof} 

\section{Nonlinear energy estimates on $\widehat{\rho}(t,k)$}

\subsection{$L^2_k$ estimates on $\hat{\rho}$} 
From \eqref{eq:rhointegral} and the linearized damping inequality, Proposition \ref{lem:LinearL2Damping}, we have (recall $A_\sigma = \abs{k}^{1/2}\jap{k,kt}^\sigma$),  
\begin{align} 
\norm{A_{\sigma_4}\hat{\rho}}^2_{L_t^2(I)} & \lesssim \norm{A_{\sigma_4}\widehat{h_{\mbox{{\scriptsize in}}}}(k,k\cdot)}^2_{L_t^2(I)} \nonumber \\ 
& \hspace{-3cm} +  \int_0^{T_\star}\left[ A_{\sigma_4}(t,k) \int_0^t\int_{\ell} \hat{\rho}(\tau,\ell) \widehat W(\ell) \ell \cdot k(t-\tau) \hat g(\tau,k-\ell,kt-\ell\tau)  \dd\tau  \dd\ell\right]^2 \dd t. \label{ineq:Lt2ptwise}
\end{align}
To improve the $L^2_k$ estimate \eqref{ctrl:MidBoot}, we integrate in $k$ to yield, 
\begin{align} 
\norm{A_{\sigma_4}\hat{\rho}}^2_{L_t^2L_k^2(I \times \Real^3)} & \lesssim \int_k \int_0^{T_\star} \abs{ A_{\sigma_4}(k,kt) \widehat{h_{\mbox{{\scriptsize in}}}}(k,kt)}^2 \dd t \dd k  \nonumber  \\ 
& \hspace{-2cm} \quad + \int_k \int_0^{T_\star}\left[A_{\sigma_4}(t,k)\int_0^t\int_{\ell} \hat{\rho}(\tau,\ell) \widehat W(\ell) \ell \cdot k(t-\tau) \hat g(\tau,k-\ell,kt-\ell\tau)  \dd\tau  \dd\ell\right]^2 \dd t \dd k. \label{ineq:NonlinCtrl}
\end{align}
As in Lemma \ref{lem:IDconst} we have, 
\begin{align} 
\int_k \int_0^{T_\star} \abs{ \abs{k} \jap{k,kt}^{2\sigma_4} \widehat{h_{\mbox{{\scriptsize in}}}}(k,kt)}^2 \dd t \dd k & \lesssim \epsilon^2. \label{ineq:rhoL2ID} 
\end{align} 
It remains to see how to deal with the nonlinear contributions in \eqref{ineq:NonlinCtrl}. 
By the triangle inequality and \eqref{def:W}: 
\begin{align} 
& \norm{A_{\sigma_4}\hat{\rho}}^2_{L_t^2L_k^2(I \times \Real^3)} \nonumber \\
& \lesssim \epsilon^2 + \int_k \int_0^{T_\star}\left[\int_0^t\int_{\ell} \jap{k-\ell,kt-\ell \tau}^{\sigma_4} \abs{k} \abs{\hat{\rho}(\tau,\ell) \frac{\ell}{\jap{\ell}^2} \cdot k(t-\tau) \hat g(\tau,k-\ell,kt-\ell\tau)}  \dd\tau  \dd\ell\right]^2 \dd t \dd k \nonumber \\
 & \quad + \int_k \int_0^{T_\star}\left[\int_0^t\int_{\ell} \jap{\ell, \ell \tau}^{\sigma_4} \abs{k} \abs{\hat{\rho}(\tau,\ell) \frac{\ell}{\jap{\ell}^2} \cdot k(t-\tau) \hat g(\tau,k-\ell,kt-\ell\tau) }  \dd\tau  \dd\ell\right]^2 \dd t \dd k \nonumber \\
& = \epsilon^2 + T + R, \label{ineq:MomentInPf} 
\end{align}  
where we refer to $T$ and $R$ as \emph{transport} and \emph{reaction} as they are analogous to the corresponding terms named similarly in \cite{BMM13} (the terminology ``reaction'' goes back to \cite{MouhotVillani11} and ``transport'' goes back to \cite{BM13}).  

\subsubsection{Transport} 
The purpose of this section is to prove the following: 
\begin{align} 
T \lesssim  K_2K_3\epsilon^4, \label{ineq:rhoTransConc}
\end{align}
which is consistent with Proposition \ref{prop:boot} provided $\epsilon$ is chosen sufficiently small. 
By Cauchy-Schwarz, 
\begin{align*} 
T & \lesssim \int_k\int_0^{T^\star} \left[\int_0^t\int_{\ell} \abs{k(t-\tau)\jap{k-\ell,kt-\ell \tau}^{\sigma_4} \hat{g}(\tau,k-\ell,kt-\ell\tau)} \abs{k}^{1/2}\abs{\frac{\ell}{\jap{\ell}^2}  \hat \rho(\tau,\ell)} \dd\tau  \dd\ell \right]^{2} \dd t \dd k \\ 
& \lesssim \int_{k} \int_0^{T^\star} \left[\int_{\ell}\int_0^t \jap{\tau}^{5/2} \frac{\abs{\ell}}{\jap{\ell}^2} \abs{\hat{\rho}(\tau,\ell)} \dd\tau  \dd\ell \right] \\ 
& \quad\quad \times 
\left[\int_{\ell} \int_0^t\abs{k(t-\tau)\jap{k-\ell,kt-\ell \tau}^{\sigma_4} \hat{g}(\tau,k-\ell,kt-\ell\tau)}^2 \abs{k}\jap{\tau}^{-5/2} \frac{\abs{\ell}}{\jap{\ell}^2} \abs{\hat \rho(\tau,\ell)} \dd\tau  \dd\ell \right] \dd t. 
\end{align*} 
Using \eqref{ineq:dissipReq}, 
\begin{align*}  
T & \lesssim \sqrt{K_5}\epsilon \int_{k} \int_0^{T^\star} \int_{\ell} \int_0^t\abs{(\abs{k(t-\tau)}\jap{k-\ell,kt-\ell \tau}^{\sigma_4} \widehat{g}(\tau,k-\ell,kt-\ell\tau)}^2 \abs{k} \jap{\tau}^{-5/2} \frac{\abs{\ell}}{\jap{\ell}^{2}} \abs{\hat \rho(\tau,\ell)} \dd\tau  \dd\ell \dd t \dd k \\ 
& \leq \sqrt{K_5} \epsilon \int_{\ell} \int_0^{T^\star} \left(\int_{k}\abs{k}\int_{-\infty}^\infty \abs{k(t-\tau) \jap{k-\ell,kt-\ell \tau}^{\sigma_4} \widehat{g}(\tau,k-\ell,kt - \ell\tau)}^2 \dd t \dd k\right) \jap{\tau}^{-5/2} \frac{\abs{\ell}}{\jap{\ell}^2} \abs{\hat{\rho}(\tau,\ell)} \dd\tau \\ 
& \lesssim K_5 \epsilon^2\sup_{\tau \geq 0} \sup_{\ell \in \Real^3} \jap{\tau}^{-5}\left(\int_{k}\abs{k}\int_{-\infty}^\infty \abs{\abs{(kt - \ell \tau)-\tau(k-\ell)} \jap{k-\ell,kt-\ell \tau}^{\sigma_4} \widehat{g}(\tau,k-\ell,kt - \ell\tau)}^2 \dd t \dd k\right) \\ 
& = K_5\epsilon^2\sup_{\tau \geq 0} \sup_{\ell \in \Real^3} \jap{\tau}^{-5}\left(\int_{k}\int_{-\infty}^\infty \abs{\widehat{\left(\jap{t\grad_z,\grad_v}\jap{\grad_{z,v}}^{\sigma_4} g\right)}(\tau,k-\ell,\frac{k}{\abs{k}}\zeta - \ell\tau)}^2 \dd \zeta \dd k\right) \\ 
& \leq  K_5\epsilon^2\sup_{\tau \geq 0} \sup_{\ell \in \Real^3} \jap{\tau}^{-5}\left(\int_{k}\sup_{\omega\in \mathbb{S}^{2}}  \int_{-\infty}^\infty \abs{\widehat{\left(\jap{t\grad_z,\grad_v} \jap{\grad_{z,v}}^{\sigma_4} g\right)}(\tau,k-\ell,\omega\zeta - \ell\tau)}^2 \dd \zeta \dd k\right) \\ 
& \leq  K_5\epsilon^2\sup_{\tau \geq 0} \sup_{\ell \in \Real^3} \jap{\tau}^{-5}\left( \int_{k} \sup_{x\in \Real^3} \sup_{\omega\in \mathbb{S}^{2}} \int_{-\infty}^\infty \abs{\widehat{\left(\jap{t\grad_z,\grad_v} \jap{\grad_{x,v}}^{\sigma_4} g\right)}(\tau,k,\omega\zeta - x)}^2 \dd \zeta \dd k\right). 
\end{align*}  
The inside factor is an $L^2$ norm along a line in $\Real^3$,
supremumed over all possible lines, therefore, by the Sobolev trace
Lemma~\ref{lem:SobTrace}, we have from \eqref{ctrl:HiLocalizedBoot},
\begin{align*} 
  T & \lesssim K_5 \epsilon^2 \sup_{\tau \geq 0} \jap{\tau}^{-5} \sum_{\abs{\alpha} \leq M} \norm{v^\alpha \jap{t\grad_z,\grad_v} \jap{\grad}^{\sigma_4} g}^2_{2} \lesssim K_5 K_1 \epsilon^4,  
\end{align*} 
as stated in \eqref{ineq:rhoTransConc}. 
By choosing $\epsilon^2 \ll K_1 K_5$, this is consistent with Proposition \ref{prop:boot}. 

\subsubsection{Reaction} \label{sec:Ereac}
For the reaction term we will prove, 
\begin{align} 
R \lesssim  K_3\epsilon^2\norm{A_{\sigma_4}\hat{\rho}}^2_{L_t^2L^2_k}, \label{ineq:rhoReacConc}
\end{align} 
since for $\epsilon$ chosen sufficiently small, this contribution can then be absorbed on the LHS of \eqref{ineq:MomentInPf}. 
By \eqref{def:W}, 
\begin{align*} 
R  \lesssim  \int_k \int_0^{T^\star} \left[\int_0^t\int_{\ell} \abs{k}^{1/2}\abs{\hat{g}(\tau,k-\ell,kt-\ell\tau)\frac{\abs{\ell \cdot k(t-\tau)}}{\jap{\ell}^2} \jap{\ell, \ell \tau}^{\sigma_4} \hat{\rho}(\tau,\ell)}  \dd\ell \dd\tau  \right]^{2} \dd t \dd k. 
\end{align*} 
By Schur's test, it follows that 
\begin{align} 
R & \lesssim \left(\sup_{t \geq 0} \sup_{k \in \Real^3} \int_0^t\int_{\ell} \bar{K}(t,\tau,k,\ell)  \dd\ell \dd\tau\right) \left(\sup_{\tau \geq 0} \sup_{\ell\in \Real^3} \int_{k} \int_{\tau}^{T^\star}\bar{K}(t,\tau,k,\ell) \dd k \dd t\right) \norm{A_{\sigma_4}\hat{\rho}}^2_{L^2_t L_k^2}, 
\end{align} 
where the time-response kernel $\bar{K}(t,\tau,k,\ell)$ is given in \eqref{def:trK}. 
Therefore, Lemmas \ref{lem:Moment} and \ref{lem:dual} imply \eqref{ineq:rhoReacConc}. 
Putting the previous estimates together and choosing $\epsilon$ sufficiently small implies \eqref{ctrl:MidBoot}. 

\subsection{$L^\infty_k L^2_t$ estimate on $\hat{\rho}$} 
Next, it remains to see how we can get the requisite
$L^\infty_k L_t^2$ estimate on $\hat{\rho}$.  For this, we will rely
on the higher regularity controls \eqref{ctrl:MidBoot}, and
\eqref{ctrl:LoLocalizedBoot}.  
Fix $k$ arbitrary. As in \eqref{ineq:Lt2ptwise} above,
applying Lemma \ref{lem:LinearL2Damping} to \eqref{eq:rhointegral}
implies
\begin{align*}
\norm{A(\cdot,k)_{\sigma_2} \hat{\rho}(\cdot,k)}^2_{L_t^2(I)} & \lesssim \norm{\abs{k}^{1/2}\jap{k,k\cdot}^{\sigma_2} \widehat{h_{\mbox{{\scriptsize in}}}}(k,k\cdot)}^2_{L_t^2(I)} \nonumber \\ 
&  +  \int_0^{T_\star}\left[ \abs{k}^{1/2}\jap{k,kt}^{\sigma_2} \int_0^t\int_{\ell} \hat{\rho}(\tau,\ell) \widehat W(\ell) \ell \cdot k(t-\tau) \hat g(\tau,k-\ell,kt-\ell\tau)  \dd\tau  \dd\ell\right]^2 \dd t \\ 
& = L(k) + NL(k). 
\end{align*}
From \eqref{lem:IDconst} it follows that $L(k) \lesssim \epsilon^2$, so it suffices to consider the nonlinear term.
We begin by dividing into two contributions via the triangle inequality:  
\begin{align*}
NL  & \lesssim \int_0^\infty \left(\abs{k}^{1/2} \int_0^t \int_\ell \Big[\jap{\ell,\ell \tau}^{\sigma_2} + \jap{k-\ell,kt - \ell \tau}^{\sigma_2}\Big] \abs{\hat{\rho}(\tau,\ell) \ell \widehat{W}(\ell) \cdot k(t-\tau) \hat{g}(k-\ell,kt-\ell \tau)}  \dd\ell \dd\tau\right)^{2} \dd t \\ 
& \lesssim R + T. 
\end{align*}
For the $R$ term we start with Cauchy-Schwarz followed by \eqref{ctrl:MidBoot} and \eqref{ctrl:LowLinftyf}, for some $\alpha > 6$ depending on $\sigma_i$:  
\begin{align*}
R & \lesssim \int_0^\infty \left(\int_0^t \int_{\ell} \jap{\ell,\ell \tau}^{2\sigma_4} \abs{\ell} \abs{\hat{\rho}(\tau,\ell)}^2  \dd\ell \dd\tau\right) 
\left(\int_0^t \int_{\ell} \frac{\abs{\ell}}{\jap{\ell}^4} \frac{\abs{k}^{3} \abs{t-\tau}^2}{\jap{\ell,\ell \tau}^{2\sigma_4-2\sigma_2}} \abs{\hat{g}(k-\ell,kt-\ell \tau)}^2  \dd\ell \dd\tau\right) \dd t \\
& \lesssim K_2K_5\epsilon^4 \int_0^\infty \int_0^t \int_{\ell} \frac{\abs{\ell}}{\jap{\ell}^4} \frac{\abs{k}^{3} \abs{t-\tau}^2}{\jap{\ell,\ell \tau}^{2\sigma_4-2\sigma_2} \jap{k-\ell, kt -\ell \tau}^{2\sigma_1}}   \dd\ell \dd\tau  \dd t \\ 
& \lesssim  K_2K_5\epsilon^4 \int_0^\infty \int_0^t\frac{\abs{k}^{3} \abs{t-\tau}^2}{\jap{\tau}\jap{k}^4\jap{kt}^{\alpha/2}} \left(\int_\ell \frac{1}{\jap{\ell}^{3}\jap{\ell,\ell \tau}^{\alpha/2}}  \dd\ell\right) \dd\tau \dd t \\ 
& \lesssim K_2K_5\epsilon^4 \int_0^\infty \int_0^t\frac{\abs{k}^{3} \abs{t-\tau}^2}{\jap{\tau}\jap{k}^4\jap{kt}^{\alpha/2}} \frac{1}{\jap{\tau}^{3-2\delta}} \dd\tau \dd t \\ 
& \lesssim K_2K_5\epsilon^4.  
\end{align*}
Consider next $T$, using $\hat{\rho}(\tau,\ell) = \hat{g}(\tau, \ell, \ell \tau)$ and \eqref{ctrl:LowLinftyf}, followed by Cauchy-Schwarz in $\ell$ and \eqref{ctrl:LoLocalizedBoot}, again for some $\alpha > 3$ depending on the gaps between $\sigma_i$: 
\begin{align*}
T & \lesssim K_5\epsilon^2\int_0^\infty \left(\int_0^t \int_\ell \frac{\abs{k}^{1/2}\abs{\ell}}{\jap{\ell}^2 \jap{\ell,\ell \tau}^{\sigma_1}} \abs{k(t-\tau)}\jap{k-\ell,kt - \ell \tau}^{\sigma_2}\abs{\hat{g}(\tau,k-\ell,kt-\ell \tau)}  \dd\ell \dd\tau \right)^2 \dd t \\ 
& \lesssim K_5K_3\epsilon^4 \int_0^\infty \frac{\abs{k}}{\jap{kt}^{\alpha}\jap{k}} \left(\int_0^t \jap{\tau} \left(\int_\ell \frac{\abs{\ell}^{2}}{\jap{\ell}^3 \jap{\ell,\ell \tau}^{\alpha}} \frac{1}{\abs{k-\ell}^{2\delta}}  \dd\ell \right)^{1/2} \dd\tau \right)^2 \dd t \\ 
& \lesssim K_5K_3\epsilon^4 \int_0^\infty \frac{\abs{k}}{\jap{kt}^{\alpha}\jap{k}} \left(\int_0^t \jap{\tau} \left(\frac{1}{\jap{\tau}^{5-4\delta}}\right)^{1/2} \dd\tau \right)^2 \dd t \\ 
& \lesssim K_5K_3\epsilon^4 \int_0^\infty \frac{\abs{k}}{\jap{kt}^{\alpha}\jap{k}} \dd t \\ 
& \lesssim K_5K_3\epsilon^4.
\end{align*}
This completes the estimate $L^\infty_k L^2_t$ estimate on $\hat{\rho}$. 

\section{Nonlinear energy estimates on $g$}

\subsection{High norm estimates} \label{sec:Hi}
\subsubsection{Estimate on $\norm{\jap{\grad_v}g}_{H_M^{\sigma_4}}$} \label{sec:Hiv}
First consider the velocity high norm estimate on $g$ stated in Proposition \ref{prop:boot}.  
Let $\alpha \in \mathbb{N}^3$ be a multi-index.  
An energy estimate yields 
\begin{align*}
\frac{1}{2}\dt\norm{\jap{\grad_v} v^\alpha g}_{H_M^{\sigma_4}}^2 & = -\int_{k,\eta} \jap{\eta}\jap{k,\eta}^{\sigma_4} D_\eta^\alpha \hat{g}(k,\eta)\jap{\eta}\jap{k,\eta}^{\sigma_4}\left(\hat{\rho}(t,k) \widehat{W}(k) k \cdot D_\eta^\alpha\left( (\eta-kt) \ff^0(\eta-kt)\right) \right) \dd k \dd\eta \\ 
  & \hspace{-2cm} - \int_{k,\eta,\ell} \jap{\eta}\jap{k,\eta}^{\sigma_4} D_\eta^\alpha \hat{g}(k,\eta)\jap{\eta}\jap{k,\eta}^{\sigma_4}\left(\hat{\rho}(t,\ell) \widehat{W}(\ell) \ell \cdot D_\eta^\alpha \left( (\eta-kt) \hat{g}(k-\ell,\eta-\ell t) \right) \right)  \dd\ell \dd k \dd\eta \\ 
& = L + NL. 
\end{align*}
Consider first the linear term $L$. We have 
\begin{align*}
L & \lesssim \int_{k,\eta} \jap{\eta}\jap{k,\eta}^{\sigma_4} \abs{D_\eta^\alpha \hat{g}(k,\eta)}\jap{kt}\jap{k,kt}^{\sigma_4}\frac{\abs{k}}{\jap{k}^2}\abs{\hat{\rho}(t,k)} \jap{\eta-kt}^{\sigma_4} \abs{D_\eta^\alpha\left((\eta-kt)\ff^0(\eta-kt)\right)} \dd k \dd\eta \\ 
& \lesssim \int_{k,\eta} \jap{\eta}\jap{k,\eta}^{\sigma_4} \abs{D_\eta^\alpha \hat{g}(k,\eta)} \jap{k,kt}^{\sigma_4} t\abs{k}^{1/2} \abs{\hat{\rho}(t,k)} \jap{\eta-kt}^{\sigma_4} \abs{D_\eta^\alpha\left((\eta-kt)\ff^0(\eta-kt)\right)} \dd k \dd\eta \\ 
& \lesssim \jap{t}\norm{\jap{\grad_v}v^\alpha g}_{H^{\sigma_4}_0} \norm{A_{\sigma_4} \hat{\rho}}_{L^2_k} \\ 
& \lesssim \frac{\delta'}{\jap{t}}\norm{\jap{\grad_v}v^\alpha g}_{H^{\sigma_4}_0}^2 + \frac{\jap{t}^3}{\delta'}\norm{A_{\sigma_4} \hat{\rho}}_{L^2_k}^2,
\end{align*}    
which for $\delta'$ sufficiently small depending only on universal constants and $f^0$ and $K_1$ sufficiently large depending only on $\delta'$ and $K_2$ is consistent with Proposition \ref{prop:boot}. 

Turn next to the nonlinear term $NL$. It is here where the full $\jap{t}^{5/2}$ growth is observed. 
First, we commute the moments and the differentiation in the transport operator: 
\begin{align}
NL & = -\int \jap{\grad_v} \jap{\grad}^{\sigma_4} \left(v^\alpha g\right) \jap{\grad_v} \jap{\grad}^{\sigma_4}\left(E(t,z+tv,v) \cdot (\grad_v - t\grad_z) (v^\alpha g) \right) \dd v \dd z \nonumber \\ 
& \quad - \int \jap{\grad_v} \jap{\grad}^{\sigma_4} \left(v^\alpha g\right) \jap{\grad_v} \jap{\grad}^{\sigma_4}\left(E(t,z+tv,v) \cdot \grad_v(v^\alpha) g \right) \dd v \dd z \nonumber \\ 
& = NL_0 + NL_M. \label{def:NLMHiV}
\end{align} 
First consider the leading order $NL_0$ term. As is standard when treating transport equations, we use the following integration by parts trick: 
\begin{align*}
NL_0 & = -\int \jap{\grad_v} \jap{\grad}^{\sigma_4} \left(v^\alpha
       g\right) \Big[ \jap{\grad_v}
       \jap{\grad}^{\sigma_4}\left(E(t,z+tv,v) \cdot (\grad_v -
       t\grad_z) (v^\alpha g) \right) \\ & \quad\quad 
                                                   - E(t,z+tv,v) \cdot
                                                   (\grad_v -
                                                   t\grad_z)\jap{\grad_v}
                                                   \jap{\grad}^{\sigma_4}
                                                   (v^\alpha g)
                                                   \Big] \dd v \dd z.
\end{align*}
On the frequency side this becomes the following, which we decompose based on which frequencies are dominant: 
\begin{align*} 
NL_0 & = -\int_{k,\eta,\ell} \jap{\eta}\jap{k,\eta}^{\sigma_4} D_\eta^\alpha \overline{\hat{g}}(k,\eta) \left(\jap{\eta}\jap{k,\eta}^{\sigma_4} - \jap{\eta-t\ell}\jap{k-\ell,\eta-t\ell}^{\sigma_4}\right) \\ & \quad\quad \times \left(\hat{\rho}(t,\ell) \widehat{W}(\ell) \ell \cdot (\eta-kt) D_\eta^\alpha \hat{g}(k-\ell,\eta-\ell t) \right)  \dd\ell \dd k \dd\eta \\
& = -\int_{k,\eta,\ell}\left(\mathbf{1}_{\abs{\ell,\ell t} \geq \abs{k-\ell, \eta-t\ell}} + \mathbf{1}_{\abs{\ell,\ell t} \leq \abs{k-\ell, \eta-t\ell}} \right) \jap{\eta}\jap{k,\eta}^{\sigma_4} D_\eta^\alpha \overline{\hat{g}}(k,\eta) \\ & \quad\quad \times \left(\jap{\eta}\jap{k,\eta}^{\sigma_4} - \jap{\eta-t\ell}\jap{k-\ell,\eta-t\ell}^{\sigma_4}\right) \left(\hat{\rho}(t,\ell) \widehat{W}(\ell) \ell \cdot (\eta-kt) D_\eta^\alpha \hat{g}(k-\ell,\eta-\ell t) \right)  \dd\ell \dd k \dd\eta \\ 
& = R + T,  
\end{align*} 
where the terminology is again ``reaction'' and ``transport'' in analogy with \eqref{ineq:MomentInPf} (and \cite{BM13}). 
Consider the leading order $R$, 
\begin{align*} 
R & \lesssim \int_{k,\eta,\ell} \mathbf{1}_{\abs{\ell,\ell t} \geq \abs{k-\ell, \eta-t\ell}} \jap{\eta}\jap{k,\eta}^{\sigma_4} \abs{D_\eta^\alpha \hat{g}(k,\eta)} \jap{\ell t}\jap{\ell,\ell t}^{\sigma_4} \abs{\hat{\rho}(t,\ell)} \frac{\abs{\ell}}{\jap{\ell}^2} \abs{\eta-kt} \abs{\hat{g}(k-\ell,\eta-\ell t)}   \dd\ell \dd k \dd\eta \\ 
& \lesssim \int_{k,\eta,\ell} \mathbf{1}_{\abs{\ell,\ell t} \geq \abs{k-\ell, \eta-t\ell}} \jap{\eta}\jap{k,\eta}^{\sigma_4} \abs{D_\eta^\alpha \hat{g}(k,\eta)} \frac{\jap{\ell t} \abs{\ell t}}{\jap{\ell}^2}\jap{\ell,\ell t}^{\sigma_4} \abs{\hat{\rho}(t,\ell)} \abs{\widehat{(\jap{\grad}g)}(k-\ell,\eta-\ell t)}   \dd\ell \dd k \dd\eta \\ 
& \lesssim \epsilon\jap{t}^2\norm{\jap{\grad_v} v^\alpha g}_{H^{\sigma_4}_0} \norm{A_{\sigma_4} \hat{\rho}}_{L^2_k} \\ 
& \lesssim \frac{\delta'}{\jap{t}}\norm{\jap{\grad_v}v^\alpha g}_{H^\sigma_0}^2 + \epsilon^2\frac{\jap{t}^5}{\delta'}\norm{A_{\sigma_4} \hat{\rho}}_{L^2_k}^2,
\end{align*}
which for $\delta'$ sufficiently small depending only on universal constants and $\epsilon$ sufficiently small, is consistent with Proposition \ref{prop:boot}. 
This completes the treatment of the reaction term $R$. 

Turn next to the transport term. We use the following inequality, which follows from the mean value theorem and holds on the support of the integrand:  
\begin{align*}
\jap{\eta}\jap{k,\eta}^\sigma - \jap{\eta-t\ell}\jap{k-\ell,\eta-t\ell}^\sigma & =  \jap{\eta}\left(\jap{k,\eta}^\sigma - \jap{k-\ell,\eta-t\ell}^\sigma\right) + \left(\jap{\eta} - \jap{\eta-t \ell} \right)\jap{k-\ell,\eta-t\ell}^\sigma \\ 
& \lesssim  \jap{\eta}\jap{k-\ell,\eta-t\ell}^{\sigma-1}\abs{\ell,\ell t} + \jap{t\ell}\jap{k-\ell,\eta-t\ell}^\sigma \\ 
& \lesssim \jap{\ell,\ell t}^2\left(\jap{\eta - t\ell}\jap{k-\ell,\eta-t\ell}^{\sigma-1} + \jap{k-\ell,\eta-t\ell}^\sigma\right). 
\end{align*}
Applying this to $T$ implies 
\begin{align*}
T & \lesssim \int_{k,\eta,\ell}\mathbf{1}_{\abs{\ell,\ell t} \leq \abs{k-\ell, \eta-t\ell}} \jap{\eta}\jap{k,\eta}^{\sigma_4} \abs{D_\eta^\alpha \hat{g}(k,\eta)}\left(\jap{\eta - t\ell}\jap{k-\ell,\eta-t\ell}^{\sigma_4-1} + \jap{k-\ell,\eta-t\ell}^{\sigma_4}\right) \\ 
& \quad\quad \times \abs{\eta-kt} \abs{D_\eta^\alpha \hat{g}(k-\ell,\eta-\ell t)} \abs{\ell}\jap{\ell,\ell t}^2\abs{\hat{\rho}(t,\ell)}  \dd\ell \dd k \dd\eta \\
& = T_{1} + T_{2}.  
\end{align*}
For $T_{1}$ we use \eqref{ineq:dissipReq}, 
\begin{align*}
T_{1} & \lesssim \int_{k,\eta,\ell} \jap{\eta}\jap{k,\eta}^{\sigma_4} \abs{D_\eta^\alpha \hat{g}(k,\eta)}\jap{\eta - t\ell}\jap{k-\ell,\eta-t\ell}^{\sigma_4}\abs{D_\eta^\alpha \hat{g}(k-\ell,\eta-\ell t)} \abs{\ell} \jap{t}\jap{\ell,\ell t}^2\abs{\hat{\rho}(t,\ell)}  \dd\ell \dd k \dd\eta \\
& \lesssim \norm{\jap{\grad_v}v^\alpha g}^2_{H_M^\sigma} \int_\ell \abs{\ell} \jap{t}\jap{\ell,\ell t}^2\abs{\hat{\rho}(t,\ell)}  \dd\ell \\ 
& \lesssim \frac{\epsilon}{\jap{t}^{3}}\norm{\jap{\grad_v} g}_{H_M^\sigma}^2.  
\end{align*}
For $T_{2}$ we similarly use, 
\begin{align*}
T_{2} & \lesssim \int_{k,\eta,\ell} \jap{\eta}\jap{k,\eta}^\sigma \abs{D_\eta^\alpha \hat{g}(k,\eta)}\jap{k-\ell,\eta-t\ell}^{\sigma}\jap{t(k-\ell),\eta-\ell t}\abs{D_\eta^\alpha \hat{g}(k-\ell,\eta-\ell t)} \abs{\ell}\jap{\ell,\ell t}^2\abs{\hat{\rho}(t,\ell)}  \dd\ell \dd k \dd\eta \\
& \lesssim \frac{\epsilon}{\jap{t}^{4}}\norm{\jap{\grad_v} g}_{H_M^\sigma}\norm{\jap{t\grad_z,\grad_v}g}_{H_M^\sigma},
\end{align*}
which for $\eps$ sufficiently small is consistent with Proposition \ref{prop:boot}. This completes the leading order $T$ term. 

Turn next to the moment term $NL_M$ (recall \eqref{def:NLMHiV}), which we divide into high and low frequency contributions similar to $NL_0$: 
\begin{align*}
NL_M & \lesssim \sum_{\abs{\beta} = \abs{\alpha} - 1}\int_{k,\eta,\ell}\left(\mathbf{1}_{\abs{\ell,\ell t} \geq \abs{k-\ell, \eta-t\ell}} + \mathbf{1}_{\abs{\ell,\ell t} \leq \abs{k-\ell, \eta-t\ell}} \right) \jap{\eta}\jap{k,\eta}^{\sigma_4} \abs{D_\eta^\alpha \hat{g}(k,\eta)} \\ & \hspace{4cm} \times \jap{\eta}\jap{k,\eta}^{\sigma_4} \abs{\hat{\rho}(t,\ell) \widehat{W}(\ell) \ell D_\eta^{\beta} \hat{g}(k-\ell,\eta-\ell t) }  \dd\ell \dd k \dd\eta \\
& = R_M + T_M.   
\end{align*}
The $R_M$ term is treated in essentially the same way as $R$ above: 
\begin{align*}
  R_M & \lesssim \sum_{\abs{\beta} = \abs{\alpha} - 1} \int_{k,\eta,\ell} \mathbf{1}_{\abs{\ell,\ell t} \geq \abs{k-\ell, \eta-t\ell}} \jap{\eta}\jap{k,\eta}^{\sigma_4} \abs{D_\eta^\alpha \hat{g}(k,\eta)} \\ & \quad \quad \times \jap{\ell t}\jap{\ell,\ell t}^{\sigma_4} \abs{\hat{\rho}(t,\ell)} \frac{\abs{\ell}}{\jap{\ell}^2} \abs{D_\eta^{\beta} \hat{g}(k-\ell,\eta-\ell t)}   \dd\ell \dd k \dd\eta \\ 
      & \lesssim \epsilon\jap{t}\norm{\jap{\grad_v}v^\alpha g}_{H^{\sigma_4}_0} \norm{A_{\sigma_4} \hat{\rho}}_{L^2_k} \\ 
      & \lesssim \frac{\delta'}{\jap{t}}\norm{\jap{\grad_v}v^\alpha g}_{H^{\sigma_4}_0}^2 + \frac{\epsilon^2\jap{t}^3}{\delta'}\norm{A_{\sigma_4} \hat{\rho}}_{L^2_k}^2,
\end{align*}
which for $\delta'$ sufficiently small depending only on universal
constants and $\epsilon$ sufficiently small, is consistent with an improvement to \eqref{ctrl:HiLocalizedBoot}.  
For the $T_M$ treatment, we use the following, applying \eqref{ineq:dissipReq},
\begin{align*}
T_M & \lesssim \sum_{\abs{\beta} \leq \abs{\alpha}-1} \int_{k,\eta,\ell} \jap{\eta}\jap{k,\eta}^{\sigma_4} \abs{ D_\eta^\alpha \hat{g}(k,\eta)} \frac{\abs{\ell}}{\jap{\ell}^2}\abs{\hat{\rho}(t,\ell)}  \jap{\eta-t\ell}\jap{k-\ell,\eta-t\ell}^{\sigma_4} \abs{D_\eta^\beta \hat{g}(k-\ell,\eta-\ell t)}  \dd\ell \dd k \dd\eta \\ 
& \lesssim \frac{\epsilon}{\jap{t}^{4}}\norm{\jap{\grad_v} g}_{H_M^{\sigma_4}}\norm{\jap{t\grad_z,\grad_v}g}_{H_M^{\sigma_4}},  
\end{align*} 
which is sufficient as in $T_2$ above. 
This completes the estimate on $\norm{\jap{\grad_v} g}_{H_M^\sigma}$. 

\subsubsection{Estimate on $\norm{\jap{\grad_z}g}_{H_M^{\sigma_4}}$}
Next turn to the estimate on $\norm{\jap{\grad_z} g}_{H_M^{\sigma_4}}$, which proceeds similarly to that on $\norm{\jap{\grad_v} g}_{H^{\sigma_4}_M}$. 
Let $\alpha \in \mathbb{N}^3$ be a multi-index. 
An energy estimate yields 
\begin{align*}
& \frac{1}{2}\dt\norm{\jap{\grad_z} v^\alpha g}_{H_M^{\sigma_4}}^2 \\
& = -\int_{k,\eta} \jap{k}\jap{k,\eta}^{\sigma_4} D_\eta^\alpha \hat{g}(k,\eta)\jap{k}\jap{k,\eta}^{\sigma_4}\left(\hat{\rho}(t,k) \widehat{W}(k) k \cdot D_\eta^\alpha\left( (\eta-kt) \ff^0(\eta-kt)\right) \right) \dd k \dd\eta \\ 
  &  - \int_{k,\eta,\ell} \jap{k}\jap{k,\eta}^{\sigma_4} D_\eta^\alpha \hat{g}(k,\eta)\jap{k}\jap{k,\eta}^{\sigma_4}\left(\hat{\rho}(t,\ell) \widehat{W}(\ell) \ell \cdot D_\eta^\alpha \left( (\eta-kt) \hat{g}(k-\ell,\eta-\ell t) \right) \right)  \dd\ell \dd k \dd\eta \\ 
& = L + NL. 
\end{align*} 
The linear term is treated as in \S\ref{sec:Hiv}: 
\begin{align*}
L & \lesssim \int_{k,\eta} \jap{k}\jap{k,\eta}^{\sigma_4} \abs{D_\eta^\alpha \hat{g}(k,\eta)}\jap{k}\jap{k,kt}^{\sigma_4}\frac{\abs{k}}{\jap{k}^2}\abs{\hat{\rho}(t,k)} \jap{\eta-kt}^{\sigma_4} \abs{D_\eta^\alpha\left((\eta-kt)\ff^0(\eta-kt)\right)} \dd k \dd\eta \\ 
& \lesssim \int_{k,\eta} \jap{k}\jap{k,\eta}^{\sigma_4} \abs{D_\eta^\alpha \hat{g}(k,\eta)} \jap{k,kt}^{\sigma_4} \abs{k}^{1/2} \abs{\hat{\rho}(t,k)} \jap{\eta-kt}^{\sigma_4} \abs{D_\eta^\alpha\left((\eta-kt)\ff^0(\eta-kt)\right)} \dd k \dd\eta \\ 
& \lesssim \frac{\delta'}{\jap{t}}\norm{\jap{\grad_v}v^\alpha g}_{H^{\sigma_4}_0}^2 + \frac{\jap{t}}{\delta'}\norm{A_{\sigma_4} \hat{\rho}}_{L^2_k}^2,
\end{align*}
which for $\delta'$ sufficiently small depending only on universal
constants and $f^0$, and $K_1$ sufficiently large depending only on
$\delta'$, and $K_2$ is consistent with an improvement on
\eqref{ctrl:LoLocalizedBoot}.

As above in \S\ref{sec:Hiv},  we commute the moments and the differentiation in the transport operator: 
\begin{align}
NL & = -\int \jap{\grad_z} \jap{\grad}^{\sigma_4} \left(v^\alpha g\right) \jap{\grad_z} \jap{\grad}^{\sigma_4}\left(E(t,z+tv,v) \cdot (\grad_v - t\grad_z) (v^\alpha g) \right) \dd v \dd z \nonumber \\ 
& \quad - \int \jap{\grad_z} \jap{\grad}^{\sigma_4} \left(v^\alpha g\right) \jap{\grad_z} \jap{\grad}^{\sigma_4}\left(E(t,z+tv,v) \cdot \grad_v(v^\alpha) g \right) \dd v \dd z \nonumber \\ 
& = NL_0 + NL_M. \label{def:NLMHiZ}
\end{align} 
First consider the leading order $NL_0$ term, which we begin as above with an integration by parts and sub-divide based on which frequencies are dominant: 
\begin{align*}
NL_0 & = -\int \jap{\grad_z} \jap{\grad}^{\sigma_4} \left(v^\alpha
       g\right) \left[ \jap{\grad_z}
       \jap{\grad}^{\sigma_4}\left(E(t,z+tv,v) \cdot (\grad_v -
       t\grad_z) (v^\alpha g) \right) \right. \\ & \quad\quad \left.
                                                   - E(t,z+tv,v) \cdot
                                                   (\grad_v -
                                                   t\grad_z)\jap{\grad_z}
                                                   \jap{\grad}^{\sigma_4}
                                                   (v^\alpha g)
                                                   \right] \dd v \dd z \\ 
& = -\int_{k,\eta,\ell}\left(\mathbf{1}_{\abs{\ell,\ell t} \geq \abs{k-\ell, \eta-t\ell}} + \mathbf{1}_{\abs{\ell,\ell t} \leq \abs{k-\ell, \eta-t\ell}} \right) \jap{k}\jap{k,\eta}^{\sigma_4} D_\eta^\alpha \overline{\hat{g}}(k,\eta) \\ & \quad\quad \times \left(\jap{k}\jap{k,\eta}^{\sigma_4} - \jap{k-\ell}\jap{k-\ell,\eta-t\ell}^{\sigma_4}\right) \left(\hat{\rho}(t,\ell) \widehat{W}(\ell) \ell \cdot (\eta-kt) D_\eta^\alpha \hat{g}(k-\ell,\eta-\ell t) \right)  \dd\ell \dd k \dd\eta \\
& = R + T. 
\end{align*}
The reaction $R$ is treated similar to the treatment in \S\ref{sec:Hiv} (note that there is one less power of $t$ lost),  
\begin{align*} 
R & \lesssim \epsilon\jap{t}\norm{\jap{\grad_v} v^\alpha g}_{H^{\sigma_4}_0} \norm{A_{\sigma_4} \hat{\rho}}_{L^2_k} \\ 
& \lesssim \frac{\delta'}{\jap{t}}\norm{\jap{\grad_v}v^\alpha g}_{H^{\sigma_4}_0}^2 + \epsilon^2\frac{\jap{t}^3}{\delta'}\norm{A_{\sigma_4} \hat{\rho}}_{L^2_k}^2,
\end{align*}
which for $\delta'$ sufficiently small depending only on universal constants is consistent with an improvement on \eqref{ctrl:HiLocalizedBoot}. 

Similar to the analogous estimate in \S\ref{sec:Hiv}, we have
\begin{align*}
T & \lesssim \int_{k,\eta,\ell}\mathbf{1}_{\abs{\ell,\ell t} \leq \abs{k-\ell, \eta-t\ell}} \jap{k}\jap{k,\eta}^{\sigma_4} \abs{D_\eta^\alpha \hat{g}(k,\eta)}\left(\jap{k-\ell}\jap{k-\ell,\eta-t\ell}^{\sigma_4-1} + \jap{k-\ell,\eta-t\ell}^{\sigma_4}\right) \\ 
& \quad\quad \times \abs{\eta-kt} \abs{D_\eta^\alpha \hat{g}(k-\ell,\eta-\ell t)} \abs{\ell}\jap{\ell,\ell t}^2\abs{\hat{\rho}(t,\ell)}  \dd\ell \dd k \dd\eta \\
& = T_{1} + T_{2}.  
\end{align*} 
The first, $T_{1}$, is treated as above.  
\begin{align*}
T_{1} & \lesssim \int_{k,\eta,\ell} \jap{k}\jap{k,\eta}^{\sigma_4} \abs{D_\eta^\alpha \hat{g}(k,\eta)}\jap{k-\ell}\jap{k-\ell,\eta-t\ell}^{\sigma_4}\abs{D_\eta^\alpha \hat{g}(k-\ell,\eta-\ell t)} \abs{\ell} \jap{t}\jap{\ell,\ell t}^2\abs{\hat{\rho}(t,\ell)}  \dd\ell \dd k \dd\eta \\
& \lesssim \norm{\jap{\grad_z}v^\alpha g}^2_{H_0^\sigma} \int_\ell \abs{\ell} \jap{t}\jap{\ell,\ell t}^2\abs{\hat{\rho}(t,\ell)}  \dd\ell \\ 
& \lesssim \frac{\epsilon}{\jap{t}^{3}}\norm{\jap{\grad_z} f}_{H_M^\sigma}^2.  
\end{align*}
The second is treated with a slight variation: 
\begin{align*}
T_{2} & \lesssim \int_{k,\eta,\ell} \jap{k}\jap{k,\eta}^\sigma \abs{D_\eta^\alpha \hat{g}(k,\eta)}\jap{k-\ell,\eta-t\ell}^{\sigma}\jap{t(k-\ell),\eta-\ell t}\abs{D_\eta^\alpha \hat{g}(k-\ell,\eta-\ell t)} \abs{\ell}\jap{\ell,\ell t}^2\abs{\hat{\rho}(t,\ell)}  \dd\ell \dd k \dd\eta \\
& \lesssim \frac{\epsilon}{\jap{t}^{4}}\norm{\jap{\grad_z} f}_{H_M^\sigma}\norm{\jap{t\grad_z,\grad_v}f}_{H_M^\sigma},  
\end{align*}
which is still consistent with the final estimate provided $\delta < 1/2$. 

The lower order moment term $NL_M$ is treated as in \S\ref{sec:Hiv} and is omitted here for brevity. 
After collecting all the above estimates and choosing $\epsilon$ small, this completes the improvement of \eqref{ctrl:HiLocalizedBoot}. 

\subsection{The $L^\infty_t H_M^{\sigma_3}$ estimate}
In this section we improve the estimate \eqref{ctrl:LoLocalizedBoot} as stated in Proposition \ref{prop:boot}. 
For $\alpha \in \Naturals^d$, an energy estimate yields 
\begin{align*}
&\frac{1}{2}\dt \norm{\abs{\partial_z}^\delta \jap{\grad}^{\sigma_3}
  v^\alpha g}_{L^2}^2 \\
& = -\int_{k,\eta} \abs{k}^\delta \jap{k,\eta}^{\sigma_3} D_\eta^\alpha \overline{\hat{g}}(k,\eta) \abs{k}^\delta \jap{k,\eta}^{\sigma_3} \hat{\rho}(t,k) \widehat{W}(k) k \cdot D_\eta^\alpha \left((\eta-kt) \ff^0(\eta-kt)\right)  \dd k \dd\eta \\ 
  & \quad - \int_{k,\eta,\ell} \abs{k}^\delta \jap{k,\eta}^{\sigma_3} D_\eta^\alpha \overline{\hat{g}}(k,\eta) \abs{k}^\delta \jap{k,\eta}^{\sigma_3}\left(\hat{\rho}(t,\ell) \widehat{W}(\ell) \ell \cdot D_\eta^\alpha \left( (\eta-kt) \hat{g}(k-\ell,\eta-\ell t) \right) \right)  \dd\ell \dd k \dd\eta \\ 
& = L + NL. 
\end{align*}
The linear term $L$ is treated as follows,  
\begin{align*}
L & \lesssim \int_{k,\eta} \abs{k}^{\delta}\jap{k,\eta}^{\sigma_3} \abs{D_\eta^\alpha \hat{g}(k,\eta)}\jap{k,kt}^{\sigma_3}\frac{\abs{k}^{1+\delta}}{\jap{k}^2}\abs{\hat{\rho}(t,k)} \jap{\eta-kt}^{\sigma_3} \abs{D_\eta^\alpha \left((\eta-kt)\ff^0(\eta-kt)\right)} \dd k \dd\eta \\ 
L & \lesssim \int_{k,\eta} \abs{k}^{\delta}\jap{k,\eta}^{\sigma_3} \abs{D_\eta^\alpha \hat{g}(k,\eta)}\frac{\abs{k}^{1+\delta}}{\jap{k}^2\jap{k,kt}^{\sigma_4-\sigma_3}} \jap{k,kt}^{\sigma_4} \abs{\hat{\rho}(t,k)} \\ & \hspace{4cm} \times \jap{\eta-kt}^{\sigma_3} \abs{D_\eta^\alpha \left((\eta-kt)\ff^0(\eta-kt)\right)} \dd k \dd\eta \\ 
& \lesssim \frac{1}{\jap{t}^{1/2+\delta}}\int_{k,\eta} \abs{k}^{\delta} \jap{k,\eta}^{\sigma_3} \abs{D_\eta^\alpha \hat{g}(k,\eta)}\jap{k,kt}^{\sigma_4}\abs{k}^{1/2}\abs{\hat{\rho}(t,k)} \jap{\eta-kt}^{\sigma_3} \abs{D_\eta^\alpha \left((\eta-kt)\ff^0(\eta-kt)\right)} \dd k \dd\eta \\
& \lesssim \frac{1}{\jap{t}^{1/2 + \delta}} \norm{\abs{\partial_z}^{\delta} g}_{H^{\sigma_3}_M} \norm{A_{\sigma_4} \hat{\rho}}_{L^2_k},  
\end{align*}
which is sufficient to deduce
$\norm{\abs{\partial_z}^{\delta} g}_{H^{\sigma_3}_M} \lesssim_\delta
K_2\epsilon^2$
via \eqref{ctrl:MidBoot}. This is consistent with an improvement of
\eqref{ctrl:LoLocalizedBoot} by choosing $K_3$ sufficiently large relative to $K_2$ (see Remark
\ref{rmk:Ki}).

As in \S\ref{sec:Hi}, we begin the nonlinear estimate by commuting the
moments and the differentiation:
\begin{align}
NL & = -\int \abs{\partial_z}^\delta \jap{\grad}^{\sigma_3} \left(v^\alpha g\right) \abs{\partial_z}^\delta \jap{\grad}^{\sigma_3}\big[E(t,z+tv,v) \cdot (\grad_v - t\grad_z) (v^\alpha g) \big] \dd v \dd z \nonumber \\ 
& \quad - \int \abs{\partial_z}^\delta \jap{\grad}^{\sigma_3} \left(v^\alpha g\right) \abs{\partial_z}^\delta \jap{\grad}^{\sigma_3}\left(E(t,z+tv,v) \cdot \grad_v(v^\alpha) g \right) \dd v \dd z \nonumber \\ 
& = NL_0 + NL_M. \label{def:NLLow}
\end{align} 
For the leading order term, as above in \S\ref{sec:Hi}, we use the following via integration by parts and sub-dividing based on frequency: 
\begin{align*}
NL_0 & = -\int \abs{\partial_z}^\delta \jap{\grad}^{\sigma_3}
       \left(v^\alpha g\right) \left[ \abs{\partial_z}^\delta
       \jap{\grad}^{\sigma_3}\left(E(t,z+tv,v) \cdot (\grad_v -
       t\grad_z) (v^\alpha g) \right) \right. \\ & 
\hspace{4cm} \left.  - E(t,z+tv,v) \cdot (\grad_v - t\grad_z)
                                                   \abs{\partial_z}^\delta
                                                   \jap{\grad}^{\sigma_3}
                                                   (v^\alpha g)
                                                   \right] \dd v \dd z \\ 
& = -\int_{k,\eta,\ell}\left(\mathbf{1}_{\abs{\ell,\ell t} \geq \abs{k-\ell, \eta-t\ell}} + \mathbf{1}_{\abs{\ell,\ell t} \leq \abs{k-\ell, \eta-t\ell}} \right) \abs{k}^\delta  \jap{k,\eta}^{\sigma_3} D_\eta^\alpha \overline{\hat{g}}(k,\eta) \\ & \quad\quad \times \left(\abs{k}^\delta \jap{k,\eta}^{\sigma_3} - \abs{k-\ell}^\delta \jap{k-\ell,\eta-t\ell}^{\sigma_3}\right) \left(\hat{\rho}(t,\ell) \widehat{W}(\ell) \ell \cdot (\eta-kt) D_\eta^\alpha \hat{g}(k-\ell,\eta-\ell t) \right)  \dd\ell \dd k \dd\eta \\
& = R + T. 
\end{align*}
In the reaction term $R$, we use that on the support of the integrand there holds (using $\delta < 1$), 
\begin{align}
\abs{\abs{k}^\delta \jap{k,\eta}^{\sigma_3} - \abs{k-\ell}^{\delta}\jap{k-\ell,\eta-t\ell}^{\sigma_3}} & \lesssim \left(\abs{\ell}^\delta + \abs{k-\ell}^\delta\right)\jap{\ell,\ell t}^{\sigma_3}.  
\label{ineq:reac}
\end{align}
Hence, 
\begin{align*}
R & \lesssim \int_{k,\eta,\ell}\abs{k}^\delta \jap{k,\eta}^{\sigma_3} \abs{D_\eta^\alpha\hat{g}(k,\eta)} \frac{\abs{\ell}^{1+\delta}}{\jap{\ell}^2}\jap{\ell,\ell t}^{\sigma_3} \abs{\hat{\rho}(t,\ell)} \\ & \hspace{4cm} \times \left(\abs{\eta-\ell t} + t\abs{k-\ell}\right) \abs{D_\eta^\alpha \hat{g}(k-\ell,\eta-\ell t)}   \dd\ell \dd k \dd\eta \\  
& \quad + \int_{k,\eta,\ell}\abs{k}^\delta \jap{k,\eta}^{\sigma_3} \abs{D_\eta^\alpha\hat{g}(k,\eta)} \frac{\abs{\ell}}{\jap{\ell}^2}\jap{\ell,\ell t}^{\sigma_3} \abs{\hat{\rho}(t,\ell)} \\ & \hspace{4cm} \times \abs{k-\ell}^\delta\left(\abs{\eta-\ell t} + t\abs{k-\ell}\right) \abs{D_\eta^\alpha \hat{g}(k-\ell,\eta-\ell t)}   \dd\ell \dd k \dd\eta \\  
& = R_{1;V} + R_{1;Z} + R_{2;V} + R_{2;Z}, 
\end{align*}
where the sub-divisions $R_{i;V}$ vs $R_{i;Z}$ denote the terms involving $\abs{\eta-\ell t}$ and $t\abs{k-\ell}$ respectively. 
The first contribution we treat in a manner analogous to the treatment of the $L$ term above (using $\abs{\ell}^{1/2+\delta}\jap{\ell,\ell t}^{-1/2-\delta} \lesssim \jap{t}^{-1/2-\delta}$, Lemma \ref{lem:Youngs}, and Cauchy-Schwarz):   
\begin{align}
R_{1;V} & \lesssim \int_{k,\eta,\ell} \abs{k}^\delta \jap{k,\eta}^{\sigma_3} \abs{D_\eta^\alpha \hat{g}(k,\eta)} \frac{\abs{\ell}^{1/2}}{\jap{\ell}^2 \jap{t}^{1/2+\delta}} \jap{\ell,\ell t}^{\sigma_3}\abs{\hat{\rho}(t,\ell) \abs{\eta-\ell t} D_\eta^\alpha \hat{g}(k-\ell,\eta-\ell t)}   \dd\ell \dd k \dd\eta \nonumber \\
& \lesssim \frac{1}{\jap{t}^{1/2+\delta}}\norm{\abs{\partial_z}^\delta g}_{H^{\sigma_3}_M} \norm{\abs{\partial_z}^{1/2}\jap{\partial_z,\partial_z t}^{\sigma_3} \rho}_{L^2} \int_{\ell} \norm{\jap{\eta} D_\eta^\alpha \hat{g}(t,\ell,\cdot)}_{L^2_\eta}  \dd\ell \nonumber \\ 
& \lesssim \frac{1}{\jap{t}^{1/2+\delta}}\norm{\abs{\partial_z}^\delta g}_{H^{\sigma_3}_M}^2 \norm{\abs{\partial_z}^{1/2}\jap{\partial_z,\partial_z t}^{\sigma_3} \rho}_{L^2}. \label{ineq:R1V}
\end{align}
 This estimate is sufficient to improve \eqref{ctrl:LoLocalizedBoot} for $\delta > 0$ and $\epsilon$ sufficiently small by \eqref{ctrl:MidBoot}. 
Turn next to $R_{Z}$, which is treated with a slight variation (using \eqref{ctrl:LoLocalizedBoot}):  
\begin{align}
R_{1;Z} & \lesssim \int_{k,\eta,\ell} \abs{k}^\delta \jap{k,\eta}^{\sigma_3} \abs{D_\eta^\alpha \hat{g}(k,\eta)} \frac{\abs{\ell}^{1+\delta}t}{\jap{\ell}^2}\jap{\ell,\ell t}^{\sigma_3} \abs{\hat{\rho}(t,\ell) \abs{\eta-\ell t} \abs{k-\ell} D_\eta^\alpha \hat{g}(k-\ell,\eta-\ell t)}   \dd\ell \dd k \dd\eta \nonumber \\
& \lesssim t\norm{\abs{\partial_z}^\delta g}_{H^{\sigma_3}_M}^2 \int_{\ell} \frac{\abs{\ell}^{1+\delta}}{\jap{\ell}^{2}} \jap{\ell, \ell t}^{\sigma_3} \abs{\hat{\rho}(t,\ell)}  \dd\ell \nonumber \\ 
 & \lesssim  t\norm{\abs{\partial_z}^\delta g}_{H^{\sigma_3}_M}^2 \left(\int_{\ell} \frac{\abs{\ell}^{1+2\delta}}{\jap{\ell}^{4}} \jap{\ell, \ell t}^{-2\sigma_4+2\sigma_3}  \dd\ell \right)^{1/2} \norm{\abs{\partial_z}^{1/2}\jap{\partial_z,\partial_z t}^{\sigma_4} \rho}_{L^2} \nonumber \\ 
& \lesssim \frac{1}{\jap{t}^{1+\delta}}\norm{\abs{\partial_z}^\delta g}_{H^{\sigma_3}_M}^2 \norm{\abs{\partial_z}^{1/2}\jap{\partial_z,\partial_z t}^{\sigma_4} \rho}_{L^2},\label{ineq:R1Z}
\end{align}
which is sufficient to improve \eqref{ctrl:LoLocalizedBoot} for $\epsilon$ sufficiently small.  
The terms $R_{2;Z} + R_{2;V}$ are treated in the same manner as $R_{1;Z}$; the details are omitted for brevity: 
\begin{align*}
R_{1;V} + R_{1;Z} \lesssim \frac{1}{\jap{t}}\norm{\abs{\partial_z}^\delta g}_{H^{\sigma_3}_M}^2 \norm{\abs{\partial_z}^{1/2}\jap{\partial_z,\partial_z t}^{\sigma_4} \rho}_{L^2},
\end{align*}
which is sufficient to improve \eqref{ctrl:LoLocalizedBoot} for $\epsilon$ sufficiently small.  

Turn next to the transport term $T$. 
On the support of the integrand there holds (from the mean value theorem)
\begin{align*}
\abs{\abs{k}^\delta \jap{k,\eta}^{\sigma_3} - \abs{k-\ell}^{\delta}\jap{k-\ell,\eta-t\ell}^{\sigma_3}} & \lesssim \abs{k-\ell}^\delta \abs{\ell,\ell t}\jap{k-\ell,\eta-t\ell}^{\sigma_3-1} \\ & \quad + \jap{k-\ell,\eta-t\ell}^{\sigma_3}\abs{\abs{k}^\delta - \abs{k-\ell}^\delta}. 
\end{align*}
Therefore, 
\begin{align*}
T & \lesssim \int_{k,\eta,\ell}  \abs{k}^\delta \jap{k,\eta}^{\sigma_3} \abs{D_\eta^\alpha \hat{g}(k,\eta)} \frac{\abs{\ell} \abs{\ell, \ell t}}{\jap{\ell}^2}\abs{\hat{\rho}(t,\ell)}  \abs{\eta-kt} \\ & \quad\quad \times \abs{k-\ell}^\delta \jap{k-\ell,\eta-\ell t}^{\sigma_3-1} \abs{D_\eta^\alpha\hat{g}(k-\ell,\eta-\ell t)}  \dd\ell \dd k \dd\eta  \\ 
& \quad + \int_{k,\eta,\ell}\abs{k}^\delta \jap{k,\eta}^{\sigma_3} \abs{D_\eta^\alpha \hat{g}(k,\eta)} \frac{\abs{\ell}}{\jap{\ell}^2}\abs{\hat{\rho}(t,\ell)} \\ & \quad\quad \times  \abs{\eta-kt} \abs{\abs{k}^\delta - \abs{k-\ell}^\delta}\jap{k-\ell,\eta-\ell t}^{\sigma_3} \abs{D_\eta^\alpha \hat{g}(k-\ell,\eta-\ell t)}  \dd\ell \dd k \dd\eta  \\ 
& = T_1 + T_2. 
\end{align*}
First consider $T_1$. Using $\abs{\eta-kt} \lesssim \jap{t}\jap{k-\ell,\eta- t\ell}$ and \eqref{ineq:dissipReq}, there holds 
\begin{align*}
T_1 & \lesssim \int_{k,\eta,\ell} \abs{k}^\delta \jap{k,\eta}^{\sigma_3} \abs{D_\eta^\alpha \hat{g}(k,\eta)} \frac{\abs{\ell} \jap{t} \abs{\ell, \ell t}}{\jap{\ell}^2}\abs{\hat{\rho}(t,\ell)} \abs{k-\ell}^\delta \jap{k-\ell,\eta-\ell t}^{\sigma_3} \abs{D_\eta^\alpha \hat{g}(k-\ell,\eta-\ell t)}  \dd\ell \dd k \dd\eta  \\ 
& \lesssim \norm{\abs{\partial_z}^\delta  g}_{H^{\sigma_3}_M}^2 \int \frac{\abs{\ell} \jap{t} \abs{\ell,\ell t}}{\jap{\ell}^2} \abs{\rho(t,\ell)}  \dd\ell \\ 
& \lesssim \frac{\epsilon}{\jap{t}^{3}} \norm{\abs{\partial_z}^\delta g}_{H^{\sigma_3}_M}^2. 
\end{align*}
For $T_2$ we instead have the following, using $\abs{\abs{k}^\delta - \abs{k-\ell}^\delta} \leq \abs{\ell}^\delta$ and \eqref{ineq:dissipReq}, 
\begin{align*}
T_2 & \lesssim \int_{k,\eta,\ell} \abs{k}^\delta \jap{k,\eta}^{\sigma_3} \abs{D_\eta^\alpha \hat{g}(k,\eta)} \frac{\abs{\ell}^{1+\delta} }{\jap{\ell}^2}\abs{\hat{\rho}(t,\ell)}  \abs{\eta-kt}\jap{k-\ell,\eta-\ell t}^{\sigma_3} \abs{D_\eta^\alpha \hat{g}(k-\ell,\eta-\ell t)}  \dd\ell \dd k \dd\eta  \\ 
& \lesssim \norm{\abs{\partial_z}^\delta f}_{H^{\sigma_3}_M} \left(\jap{t}^{-5/2}\norm{\jap{\grad_z t,\grad_v} g}_{H^{\sigma_4}_M}\right) \jap{t}^{5/2}\int \frac{\abs{\ell}^{1+\delta}}{\jap{\ell}^2} \abs{\rho(t,\ell)}  \dd\ell \\ 
& \lesssim \frac{\epsilon}{\jap{t}^{3/2+\delta}} \norm{\abs{\partial_z}^\delta g}_{H^{\sigma_3}_M} \left(\jap{t}^{-5/2}\norm{\jap{\grad_z t,\grad_v} g}_{H^{\sigma_4}_M}\right), 
\end{align*} 
which suffices to improve \eqref{ctrl:LoLocalizedBoot} for $\epsilon$ sufficiently small by \eqref{ctrl:HiLocalizedBoot}.

Turn to the lower order moments. 
We divide the treatment into reaction and transport as above: 
\begin{align*}
NL_M & \leq \sum_{\abs{\beta} = \abs{\alpha} - 1}\int_{k,\eta,\ell}\left(\mathbf{1}_{\abs{\ell,\ell t} \geq \abs{k-\ell, \eta-t\ell}} + \mathbf{1}_{\abs{\ell,\ell t} \leq \abs{k-\ell, \eta-t\ell}} \right) \abs{k}^\delta  \jap{k,\eta}^{\sigma_3}\abs{ D_\eta^\alpha \hat{g}(k,\eta)} \\ & \hspace{4cm} \times \abs{k}^\delta  \jap{k,\eta}^{\sigma_3} \abs{\hat{\rho}(t,\ell) \widehat{W}(\ell) \ell D_\eta^\beta \hat{g}(k-\ell,\eta-\ell t)}  \dd\ell \dd k \dd\eta \\ 
& = R_M + T_M. 
\end{align*}
For the $R_M$ term, we may treat as $R_{1;V}$ and $R_{1;Z}$ above. Indeed, using \eqref{ineq:reac} and treating the resulting two terms as in \eqref{ineq:R1V} and \eqref{ineq:R1Z} respectively,
\begin{align*}
R_M & \lesssim \sum_{\abs{\beta} = \abs{\alpha} - 1}\int_{k,\eta,\ell} \abs{k}^\delta \jap{k,\eta}^{\sigma_3} \abs{D_\eta^\alpha \hat{g}(k,\eta)} \frac{\abs{\ell}^{1+\delta} + \abs{\ell}\abs{k-\ell}^\delta}{\jap{\ell}^2} \jap{\ell,\ell t}^{\sigma_4}\abs{\hat{\rho}(t,\ell)} \abs{D_\eta^\beta \hat{g}(k-\ell,\eta-\ell t)}   \dd\ell \dd k \dd\eta \\
& \lesssim \frac{1}{\jap{t}^{1/2+\delta}}\norm{\abs{\partial_z}^\delta g}_{H^{\sigma_3}_M}^2 \norm{\abs{\partial_z}^{1/2}\jap{\partial_z,\partial_z t}^{\sigma_4} \rho}_{L^2},
\end{align*}
which suffices to improve \eqref{ctrl:LoLocalizedBoot} for $\epsilon$ sufficiently small by \eqref{ctrl:MidBoot}.  
For $T_M$ we can use a treatment as in $T_2$:  
\begin{align*}
T_M & \lesssim \sum_{\abs{\beta} = \abs{\alpha} - 1} \int_{k,\eta,\ell} \abs{k}^\delta \jap{k,\eta}^{\sigma_3} \abs{D_\eta^\alpha \hat{g}(k,\eta)} \frac{\abs{\ell}^{1+\delta} }{\jap{\ell}^2}\abs{\hat{\rho}(t,\ell)} \jap{k-\ell,\eta-\ell t}^{\sigma_3} \abs{D_\eta^\beta \hat{g}(k-\ell,\eta-\ell t)}  \dd\ell \dd k \dd\eta  \\ 
& \lesssim \norm{\abs{\partial_z}^\delta g}_{H^{\sigma_3}_M} \left(\jap{t}^{-5/2}\norm{\jap{\grad_z t,\grad_v} g}_{H^{\sigma_3}_M}\right) \jap{t}^{5/2}\int \frac{\abs{\ell}^{1+\delta}}{\jap{\ell}^2} \abs{\rho(t,\ell)}  \dd\ell \\ 
& \lesssim \frac{\epsilon}{\jap{t}^{3/2+\delta}} \norm{\abs{\partial_z}^\delta g}_{H^{\sigma_3}_M} \left(\jap{t}^{-5/2}\norm{\jap{\grad_z t,\grad_v} g}_{H^{\sigma_4}_M}\right), 
\end{align*}
which, as above, is sufficient to improve \eqref{ctrl:LoLocalizedBoot} for $\epsilon$ sufficiently small. 

\subsection{The $L^\infty_t L^\infty_{k,\eta}$ estimate}
In this section we improve \eqref{ctrl:LowLinftyf}. 
Integrating \eqref{eq:feqn} gives:  
\begin{align*}
\jap{k,\eta}^{\sigma_1}\abs{\hat{g}(T,k,\eta)} & \leq
                                                 \jap{k,\eta}^{\sigma_1}\abs{\widehat{h_{\mbox{{\scriptsize
                                                 in}}}}(k,\eta)}  + \jap{k,\eta}^{\sigma_1}\int_0^T \abs{\hat{\rho}(t,k) k W(k) \cdot (\eta-kt) \hat{f^0}(\eta-kt)} \dd t  \\ 
& \quad +  \jap{k,\eta}^{\sigma_1} \int_0^T \int_\ell \abs{\hat{\rho}(t,\ell) \ell W(\ell) \cdot (\eta-k t) \hat{g}(t,k-\ell,\eta-\ell t)}  \dd\ell \dd t \\ 
& = I + L + NL.   
\end{align*}
For the linear term we have (using $|\widehat{f}^0(\eta-kt)| \lesssim \jap{\eta-kt}^{-\sigma_1-3/2}$) 
\begin{align*}
L & \lesssim \left(\int_0^T \jap{k,kt}^{2\sigma_2} \abs{\hat{\rho}(t,k)}^2 \abs{k} \dd t\right)^{1/2} \left(\int_0^T \abs{k} \jap{\eta-kt}^{2\sigma_1} \abs{W(k)(\eta-kt)}^2 \abs{\hat{f^0}(\eta-kt)}^2 \dd t\right)^{1/2} \\ 
& \lesssim  \epsilon. 
\end{align*}
For the nonlinear term, we use a more sophisticated estimate. 
Write, 
\begin{align*}
NL & \lesssim \int_0^T \int_\ell \abs{k}^\delta\left(\jap{\ell,\ell t}^{\sigma_1} + \jap{k-\ell,\eta-\ell t}^{\sigma_1}\right)\abs{\hat{\rho}(t,\ell)} \frac{\abs{\ell}}{\jap{\ell}^2} \abs{\eta-kt} \abs{\widehat{g}(t,k-\ell,\eta-\ell t)}  \dd\ell \dd t \\ 
& = NL_{HL} + NL_{LH}.  
\end{align*}
The easier is $NL_{HL}$, which is handled via the following by
\eqref{ctrl:MidLinftyrho} and \eqref{ctrl:LowLinftyf} (also using that
$\sigma_2 - \sigma_1$ and $\sigma_1$ are sufficiently large),
\begin{align*}
NL_{HL} & \lesssim \int_\ell \left( \int_0^T \abs{\rho(t,\ell)}^2 \abs{\ell} \jap{\ell,\ell t}^{2\sigma_2} \dd t \right)^{1/2} \\ & \quad\quad\quad \times \left( \int \frac{\abs{\ell} \abs{\eta-kt}^2}{\jap{\ell}^4 \jap{\ell,\ell t}^{2(\sigma_2 - \sigma_1)} \jap{k-\ell,\eta-t\ell}^{2\sigma_1}} \jap{k-\ell,\eta-t\ell}^{2\sigma_1} \abs{\widehat{g}(t,k-\ell,\eta - t\ell)}^2 \dd t \right)^{1/2}  \dd\ell \\ 
& \lesssim \epsilon^2 \int_\ell  \left( \int_0^T \frac{\abs{\ell} \abs{\eta-kt}^2}{\jap{\ell}^4 \jap{\ell,\ell t}^{2(\sigma_2 - \sigma_1)} \jap{k-\ell,\eta-t\ell}^{2\sigma_1}} \dd t \right)^{1/2}  \dd\ell \\ 
& \lesssim \epsilon^2 \int_\ell  \left( \frac{\abs{\ell}}{\abs{\ell}^3\jap{\ell}^6} \right)^{1/2}  \dd\ell \\ 
& \lesssim \epsilon^2. 
\end{align*}
Now turn to the $NL_{LH}$ term. First, we use that
$\rho(t,k) = \hat{g}(t,k,kt)$, \eqref{ctrl:LowLinftyf}, and
\eqref{ctrl:LoLocalizedBoot}; second (using that the dimension is $d=3$)
\begin{align*}
NL_{LH} & \lesssim \int_0^T \int_\ell \abs{ \hat{g}(\ell,\ell t)} \frac{\abs{\ell}}{\jap{\ell}^2} \abs{\eta-kt} \jap{k-\ell,\eta-\ell t}^{\sigma_1} \abs{\widehat{g}(k-\ell,\eta-\ell t)}  \dd\ell \dd t \\ 
& \lesssim \epsilon\int_0^T \int_\ell \frac{\abs{\ell}}{\jap{\ell,\ell t}^{\sigma_1} \jap{\ell}^2} \abs{\eta-kt} \jap{k-\ell,\eta-\ell t}^{\sigma_1} \abs{\widehat{g}(k-\ell,\eta-\ell t)}  \dd\ell \dd t \\ 
& \lesssim \epsilon\int_0^T \jap{t} \left(\int_{\ell}\frac{\abs{\ell}^{2}}{\jap{\ell,\ell t}^{2\sigma_1} \jap{\ell}^4 \abs{k-\ell}^{2\delta}}  \dd\ell \right)^{1/2}  \left(\int_\ell \abs{k-\ell}^{2\delta} \jap{k-\ell,\eta-\ell t}^{2\sigma_3} \abs{\widehat{g}(k-\ell,\eta-\ell t)}^2   \dd\ell \right)^{1/2} \dd t \\ 
& \lesssim \epsilon^2 \int_0^T \jap{t} \left(\int_{\ell}\frac{\abs{\ell}^{2}}{\jap{\ell,\ell t}^{2\sigma_1} \jap{\ell}^4 \abs{k-\ell}^{2\delta}}  \dd\ell \right)^{1/2} \dd t \\ 
& \lesssim \epsilon^2 \int_0^T \jap{t} \left( \frac{1}{\jap{t}^{2-2\delta+3}} \right)^{1/2} \dd t \\ 
& \lesssim \epsilon^2,
\end{align*}
which, by choosing $\epsilon$ sufficiently small, completes the improvement of the $L^\infty_t L_{k,\eta}^\infty$ estimate \eqref{ctrl:LowLinftyf}. 
As this is the last estimate, this also completes the proof of Proposition \ref{prop:boot}.  

\appendix 
\section{Details regarding the linear problem} \label{apx:lin}
First, we state an important lemma regarding $\mathcal{L}$. 
\begin{lemma} \label{lem:Lctrl}
Recall the definition of $\mathcal{L}$ in \eqref{def:L}. For $0 \leq j
\leq \sigma$ and any $\zeta >0$
\begin{align}
\abs{k}^j \abs{\partial_\omega^j \mathcal{L}(i\omega,k)} \lesssim_\zeta \norm{W}_{L^1} \norm{f^0}_{H^{j+3/2+\zeta}_2} \label{ineq:Lj}. 
\end{align}
\end{lemma} 
\begin{proof} 
By the regularity requirement $f^0 \in H^{\sigma+3/2+0}_{M}$ and the Sobolev trace Lemma~\ref{lem:SobTrace}, we have  
\begin{align*} 
\abs{\partial_\omega^j \mathcal{L}(i\omega,k)} & \leq \int_0^\infty \abs{kt}^j \abs{\widehat{W}(k)}\abs{k}^2 t  \abs{\widehat{f^0}(kt)} \dd t \\ 
& = \int_0^\infty \abs{\widehat{W}(k)} s^{j+1} \abs{\widehat{f^0}\left( \frac{k}{\abs{k}}s \right)} \dd s  \\ 
& \lesssim_\zeta \norm{W}_{L^1} \left(\int_0^\infty s^{2j+2} \jap{s}^{1+2\zeta} \abs{\widehat{f^0}\left( \frac{k}{\abs{k}}s \right)}^2 \dd s \right)^{1/2} \\ 
& \lesssim \norm{W}_{L^1} \norm{f^0}_{H^{j+3/2+\zeta}_2}.
\end{align*}
\end{proof} 

Next, we prove Proposition~\ref{prop:electrostatic}. 
\begin{proof}[Proof of Proposition~\ref{prop:electrostatic}] 
  Let us recall the following formula from \cite{MouhotVillani11} (essentially from \cite{Penrose}; see also \cite{Glassey94}), adapted here to our slightly different definition of $\mathcal{L}$, 
  which neatly divides $\mathcal{L}$ into real and imaginary parts:
\begin{align*}
\mathcal{L}(i\omega, k) = \widehat{W}(k)\left(\mbox{p.v.} \int
  \frac{(f_k^0)'(r)}{r - \omega\abs{k}^{-1}} \dd r - i\pi \left(f_k^0\right)'\left(\frac{\omega}{\abs{k}}\right)\right), 
\end{align*}
where 
\begin{align*}
f_k^0(r) = \int_{\frac{k}{\abs{k}}r + k_\perp} f^0(x) \dd x.  
\end{align*}
Next note that when $f^0$ is radially symmetric, $f_k^0$ does not
depend on $k$. Further, recall that if $f^0$ is radially symmetric and $f^0$ is strictly positive, then $(f^0_k)' < 0$ by $v \in \Real^3$ (see e.g. \cite{MouhotVillani11}).   
Further, observe that, by Sobolev embedding,
$f^0 \in C^{0,\gamma}$ for some $\gamma \in (0,1)$, and hence the real
part of $\mathcal{L}$ is also a $C^{0,\gamma}$ function of
$\omega \abs{k}^{-1}$ (since the Hilbert transform maps
$C^{0,\gamma} \mapsto C^{0,\gamma}$ for $\gamma \in (0,1)$ \cite{BigStein}).  Next
note that
\begin{align*}
\mbox{p.v.} \int \frac{\left(f^0_k\right)'(r)}{r} \dd r \leq 0, 
\end{align*}
and hence by the H\"older continuity, there is an $m$ depending only on $f^0$ and
$\alpha$ such that for all $\omega \abs{k}^{-1} < m$, and we deduce
\begin{align*}
\mbox{p.v.} \int \frac{\left(f^0_k\right)'(r)}{r - \omega \abs{k}^{-1}} \dd r < \frac{1}{2\alpha}.  
\end{align*}
As $0 \leq \widehat{W}(k) \leq \frac{1}{\alpha}$ (recall \eqref{def:WscrnC}), it follows that for $\omega \abs{k}^{-1} < m$, $\abs{\mathcal{L} - 1} \geq 1/2$. 

For $\omega \abs{k}^{-1} > M$, for $M$ to be chosen shortly, sufficiently large write, 
\begin{align*}
  \mbox{p.v.} \int \frac{\left(f^0_k\right)'(r)}{r - \omega \abs{k}^{-1}} \dd r & = \int_{r \leq \frac{1}{2}\omega \abs{k}^{-1}} \frac{\left(f^0_k\right)'(r)}{r - \omega \abs{k}^{-1}} \dd r + \mbox{p.v.}\int_{r > \frac{1}{2}\omega \abs{k}^{-1}} \frac{\left(f^0_k\right)'(r)}{r - \omega \abs{k}^{-1}} \dd r \\ 
                                                                     & = - \frac{|k|}{\omega}\int_{r \leq \frac{1}{2}\omega
                                                                       \abs{k}^{-1}} \left(f^0_k\right)'(r) \sum_{n=0}^\infty \left( \frac{r|k|}{\omega} \right)^n  \dd r + \mbox{p.v.}\int_{r > \frac{1}{2}\omega \abs{k}^{-1}} \frac{\left(f^0_k\right)'(r)}{r - \omega \abs{k}^{-1}} \dd r \\ 
                                                                     & = L_I + L_O. 
\end{align*}
For $f^0$ and $\grad f^0$ rapidly decaying, the outer integral satisfies, at least, 
\begin{align*}
L_O & \lesssim \frac{\abs{k}^3}{\omega^3}. 
\end{align*}
Since $\left(f^0_k\right)'$ has zero average and is rapidly decaying, the leading
order contribution to the inner integral is also decaying rapidly:
\begin{align*}
-\frac{|k|}{\omega}\int_{r \leq \frac{1}{2}\omega \abs{k}^{-1}} \left(f^0_k\right)'(r) \dd r \lesssim \frac{\abs{k}^3}{\omega^3}. 
\end{align*}   
Therefore, the next order contribution is 
\begin{align*}
L_I = -\frac{\abs{k}^{2}}{\omega^2} \int \left(f^0_k\right)'(r) r \dd r + O\left(\frac{\abs{k}^3}{\omega^3}\right)
\end{align*} 
It follows that for $\omega \abs{k}^{-1} > M$, 
\begin{align*}
\textup{Re} \, \mathcal{L}(i\omega, k) = -\frac{\abs{k}^2}{(\alpha + \abs{k}^2)\omega^2} \int \left(f^0_k\right)'(r) r \dd r + O\left(\frac{\abs{k}^3}{\omega^3}\right). 
\end{align*}
Therefore, for $M$ chosen sufficiently large (depending on $\alpha$ and $f^0$), there holds 
\begin{align*}
\abs{\textup{Re} \, \mathcal{L}(i\omega, k) - 1} \geq \frac{1}{2}.   
\end{align*}
On the other hand, for all $0 < m < M < \infty$, due to the assumptions on $f_k^0$, the imaginary part of $\mathcal{L}$ is bounded uniformly away from zero over $m < \omega \abs{k}^{-1} < M$, that is, there exists a $\kappa = \kappa(m,M) > 0$ such that 
\begin{align*}
\inf_{m \leq \omega \abs{k}^{-1} \leq M}\abs{\textup{Im} \, \mathcal{L}(i\omega,k)} \geq \kappa. 
\end{align*}
The result then follows. 
\end{proof}

Next we prove Proposition~\ref{lem:LinearL2Damping}. 
\begin{proof}[Proof of Proposition~\ref{lem:LinearL2Damping}]
Note that $\phi$ will not generally be compactly supported in time but obviously 
\begin{align*}
\norm{ \abs{k}^{\alpha} \jap{k,kt}^\sigma \phi(t,k)}_{L_t^2(I)} \leq \norm{ \abs{k}^\alpha \jap{k,kt}^\sigma \phi(t,k)}_{L_t^2(\Real_+)}. 
\end{align*}

\noindent
\textit{Step~1. A priori estimate for integer $\sigma$:} Define
\begin{align*} 
\Phi(t,k) & =  \abs{k}^\alpha \phi(t,k) \\ 
H^\prime(t,k) & = \abs{k}^\alpha H(t,k) 
\end{align*}
and multiply both sides of the equation \eqref{eq:volterra} by $\abs{k}^\alpha$ to derive
\begin{align}
\Phi(t,k) = H^\prime(t,k) + \int_0^t K^0(t-\tau,k) \Phi(\tau,k) \dd\tau.
\end{align}
If we assume \emph{a priori} that all the quantities involved are $L^2$ integrable in time, then we can take the Fourier transform in time (extending as zero for $t < 0$ and extending $H$ by zero for $t > T_\star$) and we have for $\omega \in \Real$, 
\begin{align*} 
\tilde{\Phi}(\omega,k) = \tilde{H^\prime}(\omega,k) + \tilde{K^0}(\omega,k)\tilde{\Phi}(\omega,k), 
\end{align*}
where $\tilde{\Phi}(\omega,k),\tilde{H^\prime}(\omega,k)$ and $\tilde{K^0}(\omega,k)$ 
is the Fourier transform in time of $\Phi(t,k),H^\prime(t,k)$ and $K^0(t,k)$ (respectively) after extending by zero for negative times.
Now we note that 
\begin{align}
\tilde{K^0}(\omega,k) = \mathcal{L}(i\omega,k). \label{def:Liomegak} 
\end{align}
Regularity estimates in $\omega$ imply decay in $t$, so let us prove
$H^{\sigma}$ estimates in $\omega$.  Taking $\beta$ derivatives, where
$0 \leq \beta \leq \sigma$, and multiplying by
$\abs{k}^\beta \jap{k}^\gamma$ for $0 \leq \gamma \leq \sigma$ gives,
\begin{align*} 
\abs{k}^\beta \partial_\omega^\beta \jap{k}^\gamma \tilde{\Phi}(t, \omega) = \abs{k}^\beta \jap{k}^\gamma \partial_\omega^\beta \tilde{H^\prime}(t,\omega) + \abs{k}^\beta \jap{k}^\gamma \sum_{j = 0}^\beta \begin{pmatrix} \beta \\ j \end{pmatrix} \partial_\omega^{\beta-j} \mathcal{L}(i\omega,k) \partial_\omega^j \tilde{\Phi}(t, \omega). 
\end{align*} 
By taking $L^2_\omega$ norms and using the stability condition we then have
\begin{align*} 
\abs{k}^\beta \jap{k}^\gamma \norm{\partial_\omega^\beta
  \tilde{\Phi}(\cdot,k) }_{L^2_\omega} & \lesssim_{\kappa,\beta}
                                         \abs{k}^\beta \jap{k}^\gamma
                                         \norm{\partial_\omega^\beta
                                         \tilde{H^\prime}(\cdot,k)}_{L^2_\omega}
  \\ 
  & \hspace{2cm} + \jap{k}^\gamma \sum_{j = 0}^{\beta-1} \norm{\abs{k}^{\beta-j} \partial_\omega^{\beta-j} \mathcal{L}(i\cdot,k)}_{L^\infty_\omega} \abs{k}^j\norm{\partial_\omega^j \tilde{\Phi}(\cdot,k)}_{L^2_\omega}.   
\end{align*}  
Then using \eqref{ineq:Lj} and induction on $\beta$ we get for all $\beta$, $0 \leq \beta \leq \sigma$ and all $s$ for $0 \leq s \leq \sigma$, 
\begin{align} 
\norm{\abs{kt}^{\beta} \jap{k}^s \Phi(t,k)}_{L^2_t(I)} & \lesssim_{s,\beta} \kappa^{-\beta} \jap{k}^s \norm{ \jap{kt}^{\beta} H^\prime(t,k)}_{L^2_t(I)}. \label{ineq:phictrl}
\end{align}
Now we apply
\begin{align*} 
\jap{k,kt}^\sigma \approx \jap{k}^\sigma + \abs{kt}^{\sigma}  
\end{align*} 
and use \eqref{ineq:phictrl} with $\beta = 0, s = \sigma$ and $\beta =
\sigma, s = 0$ to conclude the a priori estimate
\eqref{ineq:LinearCtrl}. 
\medskip

\noindent
\textit{Step~2. Justifying a priori estimate for integer $\sigma$:}
Recall that this argument assumes a priori that we already have
sufficiently rapid decay on $\phi$. In order to make this argument
rigorous, one may use the technique described in
\cite{BMM13,villani2010} which is for all $\delta>0$, define
$\eta_\delta(t) = e^{-\delta t^2/2}$ and choose $\mu \leq 0$ be a real
number, then study
\begin{align*}
\Phi^\delta(t,k) = e^{\mu t}\eta_\delta(t) \Phi(t,k). 
\end{align*}
It is straightforward to show that for $C$ sufficiently large $\abs{\Phi(t,k)} \lesssim e^{Ct}$ and hence for $\mu < -C$ one goes through the derivations above and derives: 
\begin{align*}
\tilde{\Phi^\delta}(\omega,k) = \tilde{\eta_\delta} \ast \left(\frac{\tilde{H}(\cdot,k)}{1 - \mathcal{L}(\mu + i\cdot,k)} \right)(\omega). 
\end{align*}
Moreover, this function depends analytically on $\mu$ as long as we
stay away from a singularity where $\mathcal{L} = 1$.  By analytic
continuation, we may hence deduce that this formula holds all the way
for all $\mu \leq 0$.  From there, one may proceed by taking
derivatives in $\omega$ on $\tilde{\Phi}^\delta(\omega,k)$ and then
passing to the limit $\delta \rightarrow 0$ to deduce the desired
estimate \eqref{ineq:LinearCtrl}.

\end{proof} 

\bibliographystyle{abbrv} \bibliography{eulereqns}

\end{document}